\documentclass{amsart}
\usepackage{amsfonts,amscd, amssymb}

\newtheorem{theorem}{Theorem}[section]
\newtheorem{lemma}[theorem]{Lemma}
\newtheorem{corollary}[theorem]{Corollary}
\newtheorem{proposition}[theorem]{Proposition}
\theoremstyle{remark}
\newtheorem{remark}[theorem]{Remark}
\theoremstyle{definition}

\numberwithin{equation}{section}
\makeatother

\DeclareMathOperator{\Cdb}{{\mathbb C}}
\DeclareMathOperator{\Rdb}{{\mathbb R}}

\DeclareMathOperator{\Hdb}{{\mathbb H}}

\DeclareMathOperator{\Ndb}{{\mathbb N}}

\begin{document}

\title[Roots in operator  and Banach algebras]{Roots
in operator  and Banach algebras}
\thanks{Supported by a grant from the NSF} 
\author{David P. Blecher}
\address{Department of Mathematics \\ University of Houston \\ Houston \\TX
77204-3008, USA}
\email{dblecher@math.uh.edu}

\author{Zhenhua Wang}
\address{Department of Mathematics \\ University of Houston \\ Houston, TX
77204-3008, USA}
\email{zhenwang@math.uh.edu}


\begin{abstract}  We show that  several  known facts concerning roots of  matrices generalize to operator algebras and Banach algebras.   We show for example that the so-called Newton,
binomial, Visser, and Halley iterative methods converge to the root in Banach and operator algebras under various 
mild hypotheses.  
We also show that the `sign' and `geometric mean' of matrices generalize
 to Banach and operator algebras, and we investigate their properties.   We also establish some other facts about roots in this setting.
  \end{abstract}

\maketitle

\dedicatory{In memoriam Charles Read--gentleman, brother,  mathematical force of nature.}
\section{Introduction} 
An {\em operator algebra} is a closed subalgebra of $B(H)$, for a
complex Hilbert space $H$. In this paper we show that  several  known facts concerning roots of  matrices generalize to 
 operator algebras and Banach algebras.     We begin by establishing some basic 
properties of roots that do not seem to be in the literature (although they may be known to some experts), 
as well as reviewing some that 
are.  We then show   that the `sign' of a matrix generalizes to Banach algebras, and that 
Drury's variant of the `geometric mean' of matrices
generalizes to operators on a Hilbert space (we also generalize his definition
slightly), and prove some basic facts about these.   We also show that the so-called  Newton (or Babylonian),
binomial, Visser, and Halley iterative methods for the root
converge to the root in Banach and operator algebras under various 
mild hypotheses inspired by the matrix theory literature.        Some parts of our paper are fairly literal 
transfers of matrix results to the operator or Banach algebraic setting, using known  
tricks or standard theory, and here we will
try to be brief.  However we have not seen these in the literature and they seem quite useful.  
For example our results, particularly probably 
the geometric mean, should be applicable to our ongoing study 
of `real positivity' in operator algebras (see e.g.\ \cite{BRI,BRII,Bord,BOZ,BSan} and references therein) initiated by the first author and Charles Read.

Turning to background and notation, it is common when studying roots to make the assumption that the spectrum contains no strictly 
negative numbers.  Note that a singular matrix with no strictly 
negative eigenvalues, may not have a square root
(for example,
$E_{12}$ in $M_2$), or may have a square root but not have 
a square root in $\{ x \}''$ (for example, 
$E_{12}$ in $M_3$, which has many square roots
including $E_{13} + E_{32}$), or may have infinitely  many 
square roots in $\{ x \}''$ (for example, $0$ in an algebra with trivial product).   However in a Banach algebra
and for $p \in \Ndb$,
any element $x$ of type $M$ (defined below; this is also sometimes  called being 
`sectorial', see e.g. \cite{Haase}, but there is ambiguity in the literature), and also any element whose  (closed)
numerical range  (defined below)
contains no strictly negative numbers,
has a unique $p$th root with spectrum in a sector $S_{\frac{\pi}{p}}$  (see \cite{Noll,LRS} and also
Theorem \ref{lrs} below), 
and this root is in the closed subalgebra generated by $x$, which in turn is a subset of the second commutant
$\{ x \}''$.
 Here
  $S_{\theta}$ is the set of complex numbers with argument in $[-\theta, \theta]$.     
Thus we will usually (but not always) 
take roots of elements  with no strictly negative numbers in its numerical range.  Indeed sometimes
we will require the  numerical range to be in $S_\theta$ for some $\theta < \pi$.

A {\em unital}  Banach algebra has an identity of norm $1$.
The {\em states} of $A$ are the norm 1 functionals $\varphi$ on $A$ 
with $\varphi(1) = 1$, they comprise the {\em state space} $S(A)$,
and the {\em numerical range} \cite{BoNR1} is
$$W(x) = \{ \varphi(x) : \varphi \in S(A) \}, \qquad x \in A .$$
This is a convex and compact set of scalars.   Some authors use not necessarily closed
versions of the numerical range, such as $\{ \langle x \zeta , \zeta \rangle : \zeta \in {\rm Ball}(H) \}$ in the case $x$ is an operator on Hilbert space $H$, but since these are dense in our $W(x)$ we avoid these.  
Let ${\mathbb H}$ denote the open right half plane, with
$\overline{{\mathbb H}}$  the closed right half plane.  We write
${\mathfrak r}_A$  for the {\em accretive} (or `real-positive')  elements in a unital Banach  algebra  $A$, i.e.\ those
elements $x$ 
with numerical range $W(x)$ in $\overline{{\mathbb H}}$.  
We say that $x$ is {\em strictly accretive} if its
numerical range is in ${\mathbb H}$.
 In a possibly nonunital operator algebra $A$ on a 
Hilbert space $H$ 
there is a unique unitization by Meyer's theorem   (see
 \cite[Section 2.1]{BLM}), which we can take to be $A + \Cdb I_H$.  Here 
we can define ${\mathfrak r}_A = A \cap {\mathfrak r}_{A^1}$, and we have 
${\mathfrak r}_A = \{ x \in A : x + x^* \geq 0 \}$.  Also, for invertible $a$,
the spectrum of $a$ is in the right half plane if and only if 
the spectrum of $a^{-1}$ is in the right half plane.
We write ${\rm Ball}(X)$ for the set $\{ x \in X : \Vert x \Vert \leq 1 \}$, and set 
$${\mathfrak F}_A = \{ x \in A : \Vert 1 - x \Vert \leq 1 \} = 1 + {\rm Ball}(A)$$
for a unital Banach  algebra  $A$.   There is an associated cone
$${\mathfrak c}_A =  \Rdb^+ {\mathfrak F}_A, $$ and we have
(see \cite{BOZ}) 
 $${\mathfrak r}_{A} = \overline{\Rdb^+ {\mathfrak F}_A}.$$ 

By a {\em root} we mean a fractional power $x^r$ where $r = \frac{1}{n}$ for $n \in \Ndb$.
See \cite[Section 6]{BSan} for a review of these.  
An element $x$ of a unital Banach algebra $A$ whose
spectrum contains no real negative numbers nor $0$,
has a unique principal $n$th root in $\{ x \}''$ for all $n \in \Ndb$;
that is a unique $n$th root with spectrum in the interior of the sector
$S_{\frac{\pi}{n}}$; hence a unique square root with spectrum in the
open right half plane ${\mathbb H}$ (see \cite[p.\ 360]{Pal} for the
square root case, which can be easily adapted for the $n$th root).
  We note that if $x$ is an element  of $A$ whose  numerical range $W(x)$ 
satisfies $W(x) \subset S_\theta$ for some $\theta < \pi$
then the formula of Stampfli and Williams  \cite[Lemma 1]{SW}
and some basic trigonometry shows that $x$ is sectorial of angle $\theta < \pi$ in the sense of e.g.\ \cite{Haase},
so that all the facts about roots of sectorial operators from that text apply.

Note that if $a$ is invertible then ${\rm Sp}(a^{-1}) = \{ \lambda^{-1} : \lambda \in {\rm Sp}(a) \}$, 
so that    we 
have $(a^{-1})^{\frac{1}{2}} = (a^{\frac{1}{2}})^{-1}$ if
Sp$(a)$ contains no real negative numbers.     This follows from  the unicity of principal roots mentioned above,
 because both
have spectrum in a sector of angle $< \frac{\pi}{2}$.

It is well known that the accretive elements are closed under roots, or $r$th powers for $r \in (0,1)$.   
 Note too that $a \in {\mathfrak c}_A$ implies that $a^r 
\in {\mathfrak c}_A$ for such $r$.   This is because  ${\mathfrak F}_A$ 
is closed under such powers (see e.g.\ \cite[Proposition 6.3]{BSan}).   Also,  in an operator algebra
if $W(a) \subset S_\theta$ for $\theta < \pi$ then $W(a^r) \subset S_{r \theta}$  for $0 \leq r \leq 1$
(see e.g.\ \cite[Corollary 4.6]{BBS}
for a more general result).

We will use Crouzeix's analytic functional calculus (e.g.\ \cite[Theorem 2.1]{Crou08}).    Crouzeix showed 
for example that there is a universal constant $K$ (we will call this Crouziex's constant), which  
 is known to be smaller than 12, and is possibly 2, such that $\Vert f(x) \Vert \leq K \,
\sup_{z \in W(x)} |f(z)|$, for any bounded operator $x$ on a Hilbert space and rational function $f$ with no 
poles in $W(x)$.

\section{More background results}

The following is no doubt well known (the formula is in Corollary 3.1.14 or 3.2.1 (d) in \cite{Haase} in the 
case $x$ is sectorial), and its proof follows a standard route.  For example, it 
 is similar to but easier than the case considered  in \cite{LRS}, but since we do not know of an explicit 
reference we
sketch the argument.

\begin{lemma}   \label{Rieszroot}   If $x$ is an invertible element in a unital Banach algebra
whose spectrum contains no real strictly negative numbers, and if $0 < \alpha < 1$, then  
$$x^{-\alpha} = \frac{\sin(\pi \alpha)}{\pi} \int_0^\infty \, t^{-\alpha} \, (t 1 + x)^{-1} \, dt.$$
In particular,
$x^{-\frac{1}{2}} = \frac{2}{\pi} \int_0^\infty \, (t^2 1 + x)^{-1} \, dt.$
\end{lemma}   \begin{proof}    
Let $R = \Vert x \Vert$, and choose $\theta$ less than but very close to $\pi$, and choose $\epsilon \geq  0$  small
enough so that $\Vert (x - z I)^{-1} - x^{-1} \Vert  <1$ for $|z| < \epsilon$.  Choose $r > R$.
Consider the simple closed curve $\Gamma^{r,\epsilon,\theta}$, oriented counterclockwise,  consisting of the numbers
with argument in $[-\theta,\theta]$ which are on the two circles centered at $0$ and having radii $r$ and $\epsilon$, together with  
the lines $z = \pm t e^{i \theta}$ for $\epsilon \leq t \leq r$.   By the Riesz functional calculus 
$$x^{-\alpha} =  \frac{1}{2 \pi i} \int_\Gamma \, (z 1 - x)^{-1} \, z^{-\alpha} \, dz .$$
The part of the integral over the small circular arc contributes something which
in norm is less than  $(\Vert x^{-1} \Vert + 1) \cdot \epsilon \cdot \epsilon^{-\alpha}$ to 
the integral.  But this converges to $0$ as $\epsilon \to 0$, and so letting $\epsilon \to 0$ we may replace
$\Gamma^{r,\epsilon,\theta}$ by $\Gamma^{r,0,\theta}$.  Looking at the bottom half of $\Gamma^{r,0,\theta}$, we may let
$\theta \to \pi^{-}$, and hence the line segment part of the curve may be taken to lie on the negative $x$ axis.  
However there is an issue with what becomes of $z^{-\alpha}$ as $z$ approaches the negative real axis from 
below: if $z = t e^{i \theta}$, for a number $\theta$ slightly larger than
$-\pi$, then $z^{-\alpha} = t^{-\alpha} e^{-i \alpha \theta} \to 
t^{-\alpha} (\cos(\alpha \pi) + i \sin(\alpha \pi))$.   Note that this is different to what happens with 
$z$ on the `upper line segment', here we will get a limit $t^{-\alpha} (\cos(\alpha \pi) - i \sin(\alpha \pi))$.
The integral over the `lower line segment' thus leads to a contribution of 
is $$\frac{-1}{2 \pi i} \int_0^r (-t 
  - x)^{-1} \, (\cos(\alpha \pi) + i \sin(\alpha \pi)) \, t^{-\alpha} \, 
dt .$$ 
Similarly, the  contribution from the `upper' line segment can be seen to be
$$\frac{1}{2 \pi i} \int_0^r (-t
  - x)^{-1} \, (\cos(\alpha \pi) - i \sin(\alpha \pi)) \, t^{-\alpha} \, 
dt , $$ and so the two line segments together contribute
$\frac{\sin(\alpha \pi)}{\pi} \int_0^r (t
  + x)^{-1} \,  t^{-\alpha} \, 
dt$.   
The circular part of $\Gamma$  is distance greater than $r - R$ from the numerical range of $x$ and so
by \cite[Lemma 1]{SW} it contributes at most  $\frac{r^{1-\alpha}}{r-R}$.  
But this converges to $0$ as $r \to \infty$.   
Thus letting 
$r \to \infty$ we obtain the desired formula.   If $\alpha = \frac{1}{2}$ we can 
let $u = \sqrt{t}$ to obtain the second formula.  
 \end{proof}

Most of the following is also  well known to experts (see e.g.\ \cite{ChLi,GR}).

\begin{lemma}   \label{inv}   If $A$ is a unital operator algebra on a Hilbert space and
if $x \in {\mathfrak r}_A$ is an invertible accretive operator in $A$ then
 $x^{-1}  \in {\mathfrak r}_A$.  That is, inverses of invertible accretive operators on a Hilbert space are accretive.  More generally, if 
$W(x) \subset S_\theta$ then $W(x^{-1}) \subset S_\theta$ if $0 \leq \theta
\leq \frac{\pi}{2}$ and $x$ is invertible.   Finally, if
$a$ is an invertible in $A$ which is  strictly accretive (this is equivalent for invertibles 
to being in ${\mathfrak c}_A$),
then $a^{-1}$ is  strictly accretive (or equivalently, in ${\mathfrak c}_A$).
 \end{lemma}   \begin{proof}  Throughout this proof let $x$ be invertible and accretive. For the first statement (which 
is well known but since it is short we will prove), suppose that $A \subset B(H)$ is a unital
subalgebra.   Then any $\eta \in H$ equals $x \zeta$ with $\zeta \in H$ and
$${\rm Re} \;\langle x^{-1}  \eta, \eta \rangle  = {\rm Re} \; \langle \zeta , x \zeta   \rangle=
{\rm Re} \; \langle x \zeta , \zeta  \rangle \geq 0 .$$
So $x^{-1}$ is accretive.    The second statement is in the references cited above the lemma. 

Now $a \in {\mathfrak c}_A$  iff there exists $t > 0$ with 
$\Vert 1 - t a \Vert^2 \leq 1$, which is easy to see via the $C^*$-identity happens iff $a + a^* \geq t a^* a$.  This in turn is equivalent to 
$a + a^* \geq \epsilon 1$ for some $\epsilon > 0$,   
since $a$ is invertible.     Then \cite[Proposition 3.5]{ChLi}  implies that $a^{-1} \in
\frac{1}{\epsilon} {\mathfrak F}_A \subset {\mathfrak c}_A$. 
\end{proof}

The last lemma is not true for unital Banach algebras.  For example in $\ell^1_2$ with the usual convolution 
product, $(1+i,1)$ is accretive, but its inverse $\frac{1}{-1+2i}(1+i,-1)$ is not accretive, using the 
criterion for being accretive given in Example 3.14  in \cite{BOZ}.

\begin{remark}  The last observation
 gives one way to see that the Cayley transform $\kappa(x)$ and the transform ${\mathfrak F}(x)$  considered e.g.\
in \cite[Section 2.2]{Bord}, are not contractions for accretive
$x$ in general unital 
Banach algebras.   Indeed if $\kappa(x)$ was contractive then ${\mathfrak F}(x) = \frac{1}{2}(1 + \kappa(x))$
is contractive, and hence 
$$\Vert (t + x^{-1})^{-1} \Vert = \frac{1}{t} {\mathfrak F}(tx) \leq  \frac{1}{t} , \qquad t > 0.$$
This implies that $x^{-1}  \in {\mathfrak r}_A$ by e.g.\ \cite[Lemma 2.4]{BSan}.
\end{remark}

We will say that an element $x$ in a unital Banach algebra $A$ is {\em type $M$} if
there exists a constant $M$ such that $\Vert (t1 + x)^{-1} \Vert \leq M/t$ for all
$t > 0$.    This is essentially what is called being sectorial in \cite{Haase} (see p.\ 20--21 there, replacing $a$ by 
left multiplication by $a$ in $B(A)$),  or sometimes called
being `non-negative'.   We use this older name simply because there is 
ambiguity in the literature: e.g.\ we (and many others) have used the word sectorial  for the stronger notion
of an operator whose numerical range is contained in a sector $S_\theta$ with $\theta < \pi$.  
Note that the  inequality $\Vert (t1 + x)^{-1} \Vert \leq 1/t$ for all
$t > 0$  is equivalent to $x$ being
accretive (see e.g.\ \cite[Lemma 2.4]{BSan}).     Inverses of invertible
type $M$ elements are type $M$ by an elementary equality for inverses \cite[Proposition 2.1.1]{Haase}.
It is well known that 
if the spectrum of an {\em  invertible} element
 $a$ contains no real strictly negative numbers then $a$ is type $M$.   
Indeed  for any  $a \in A$ the identity defining `type $M$ elements' is  always true for $t > 2 \Vert a \Vert$ by an 
inequality in the  proof of the Neumann series lemma:
$$\Vert (1 + a/t)^{-1} \Vert \leq 1/(1- \Vert  a/t \Vert) \leq 2;$$
and $t \Vert (t 1 + a )^{-1} \Vert$ is continuous and
hence bounded on $[0,2 \Vert a \Vert]$.    (We remark though that 
this is false if $a$ is not invertible.  
 Consider for example $a(z) = z$ on the parabola $y = x^2$ for $|x| \leq 1$,
here $\Vert t (t 1 + a )^{-1} \Vert$ dominates $1/t$, the absolute 
value of $t (t 1 + a )^{-1}$ at $z = -t+it^2$.  
Also, there are type $M$ elements $a$ with negative 
numbers in $W(a)$, e.g.\ most invertible $2 \times 2$ matrices with $-1$ in the $1$-$1$ entry.)

\begin{lemma}   \label{unsqr}  In a Banach algebra, if $a, b$ are  type $M$ then 
$\Vert a^t - b^t \Vert \leq K \Vert a - b \Vert^t$ for all $t \in (0,1]$, for a constant $K$ 
depending on $t$.
 \end{lemma}   \begin{proof}   This follows from the proof of the analoguous result  in \cite{MP}.
 \end{proof}

 Some details seem to be skipped in the proof of uniqueness 
in \cite[Theorem 2.8]{LRS}, which with the help of \cite{Noll}
we supply below, also slightly improving the result.   See also \cite[Chapter 6]{MS}.

\begin{theorem}  \label{lrs}   If  $A$ is a unital Banach algebra, $m \in \Ndb$, and $x \in A$ is such that $W(x)$ contains
no strictly negative numbers, then $x$ has a unique $m$th root with spectrum in $S_{\frac{\pi}{m}}$.
This root is in the closed subalgebra generated by $x$.    

 Also we have 
$(e^{i \theta} \, x)^s = e^{i s \theta} \, x^s$ for $s \in [0,1]$ and $|\theta| \leq \pi$, provided 
that $W(e^{i \rho} x)$ contains 
no strictly negative numbers for all $\rho$ between $0$ and $\theta$ (including 
$0$ and $\theta$).
\end{theorem}  
  \begin{proof}   If $W(x) \subset S_\beta$ for some $\beta < \pi$, then $x$ is type $M$ as stated 
above,  and then 
the first part of the result (except for the 
the `subalgebra generated' assertion)  is in \cite{Noll} 
(the main part being in \cite{LRS} too).
We will
take this for granted in the following argument.   In the contrary case, since $W(x)$ is convex, it follows that $W(x) \subset
i \bar{\Hdb}$ or $W(x) \subset
-i \bar{\Hdb}$.   We assume the first, the second being similar.    
Then 
$i^{\frac{1}{m}} (- i x)^{\frac{1}{m}}$
is an $m$th root of $x$
with spectrum in $i^{\frac{1}{m}} S_{\frac{\pi}{2m}} \subset S_{\frac{\pi}{m}}$.   
That $x^{\frac{1}{m}}$ is in the closed subalgebra generated by $x$ may be found e.g.\ in the discussion after
Proposition 6.3 in \cite{BSan}.

Now suppose that $c_1, c_2$ are two $m$th roots of $x$ with spectrum in $S_{\frac{\pi}{m}}$.
Then for $\epsilon > 0$ let $d_k = c_k + \epsilon 1$, then $d_k^m$ is invertible and has spectrum 
containing 
no strictly negative numbers by the spectral mapping theorem.   Thus $d_k^m$ is type $M$ by an observation
above Lemma   \ref{unsqr}, and so 
we can use the argument in \cite{LRS,Noll}: by an argument in
\cite{MP} (see Lemma \ref{unsqr} above) we have 
$$\Vert c_1 - c_2 \Vert \leq K \Vert d_1^{m} - d_2^{m} \Vert
\to 0$$
as $\epsilon \to 0$, so $c_1 = c_2$.  

For the last assertion, let $\theta$ be as described.   By writing $\theta = p \, \frac{\theta}{p}$ for a large integer
$p$ and iterating the identity we are proving $p$ times, we may assume that $\theta$ is as close to $0$ as we like.  
 In fact, the case that $- \frac{\pi}{2} \leq 
\theta < 0$ and $e^{i \theta} x$ is accretive is done in
\cite[Corollary 4.6]{BBS} (note that the first centered equation on page 564 there also 
follows from the uniqueness argument just after the next centered equation there).     Next suppose that 
the largest  argument of numbers in $W(x)$ is $\alpha > \frac{\pi}{2}$, and suppose that $\frac{\pi}{2} - \alpha < 
\theta < 0$, so that $W(e^{i \theta} x)$  still intersects the interior of the third quadrant.   Choose $\rho > 0$ 
such that $W(e^{i (\theta - \rho)} x)$ is accretive, then by the  case just discussed we have
$(e^{i (\theta - \rho)} x)^{s} = e^{i s (\theta - \rho)} x^s$, so that 
$$e^{i s \theta} x^s =  e^{i s  \rho} (e^{i (\theta - \rho)} x)^{s}  = 
(e^{i   \rho} e^{i (\theta - \rho)} x)^{s} = (e^{i \theta} x)^s,$$
where in the second last equality we used the case from \cite{BBS} again.
The next case we consider is if $x$ is accretive, and $\theta < 0$.   Let  $a = e^{i \theta} x$, then $e^{-i \theta} a = x$.
By the case from \cite{BBS} we have $(e^{-i \theta} a)^s = e^{-i s \theta} a^s$, so that
$e^{ i s \theta} x^s = (e^{i \theta} x)^s$ as desired.    Next, if $W(x)$ contains numbers in the 
interior of the third quadrant and $\theta$ negative but very small, choose $\rho > 0$ with $e^{i (\theta + \rho)} x$ accretive.   By the case from \cite{BBS}, we have
$(e^{i (\theta + \rho)} x)^{s} = e^{i s (\theta + \rho)} x^s$, so that 
$$e^{i s \theta} x^s =  e^{-i s  \rho} (e^{i (\theta + \rho)} x)^{s}  = 
(e^{-i   \rho} e^{i (\theta + \rho)} x)^{s} = (e^{i \theta} x)^s,$$
similarly to a case above.
Finally, if $\theta > 0$, replace $x$ by $a = e^{i \theta} x$ and $\theta$ with its negative, and apply the above.
 \end{proof}

The last assertion of the last result is also no doubt known to some  experts, see e.g.\ \cite[Theorem 10.1]{Kom}.
The following is no doubt also known, but again we do not know of a reference for it as stated.

\begin{corollary} \label{numraoffc}   
If $a$ is a Hilbert space operator with no strictly negative numbers in $W(a)$, and with the arguments of numbers in $W(a)$ inside
$[\alpha, \beta]$ for $-\pi \leq \alpha \leq \beta \leq \pi$, 
then  for $s \in (0,1)$ the arguments of numbers in  $W(a^s)$ are in 
$[s \alpha, s \beta]$.   \end{corollary}
\begin{proof}    Let $\nu = \frac{\beta-\alpha}{2}, \theta = \frac{\beta+\alpha}{2}$, then $W(e^{-i \theta} a)
\subset S_\nu$.   Hence using the last assertion of the last result,
$W(e^{-i s \theta} a^s) = W((e^{-i \theta} a)^s) \subset S_{s\nu}$,
so that the arguments of numbers in  $W(a^s)$ are in 
$[-s \nu +s \theta, s \nu +s \theta] =   [s \alpha, s \beta]$.
 \end{proof}

In \cite[Section 6]{BSan}  we gave an estimate for the `sectorial angle'
of $W(x^t)$ for accretive elements
in a Banach algebra.   The following is the variant of that result in the case that $W(x) \subset S_\theta$ 
for $\frac{\pi}{2} < \theta < \pi$.

\begin{lemma}   \label{touch}   If $A$ is a unital Banach
algebra 
 and
if $x \in A$ has no negative numbers in its numerical range and satisfies  
$W(x) \subset S_{\frac{\pi}{2} + \theta}$, where $0 \leq \theta \leq \frac{\pi}{2}$, 
 then $W(x^{\frac{1}{p}}) \subset S_{\frac{\pi}{2} + \frac{\theta}{p}}$ for $p \in \Ndb$. 
If $A$ is also an operator 
algebra on a Hilbert space then  $W(x^{\frac{1}{p}}) \subset S_{\frac{\pi}{2p} + \frac{\theta}{p}}$.
 \end{lemma}   \begin{proof}  We have that $e^{-i \theta} x$ is accretive,
so that $e^{-i \frac{\theta}{p}} x^{\frac{1}{p}}$ is accretive (see also the proof of
Theorem \ref{lrs}).  Hence $W(x^{\frac{1}{p}}) \subset S_{\frac{\pi}{2} + \frac{\theta}{p}}$.
The Hilbert space case is well known (see e.g.\
\cite[Theorem 2.8]{LRS} and the last section in \cite{BBS}). \end{proof}

\begin{proposition}   \label{Rsa}
 Let $t > 1$. In  a unital Banach algebra $A$ if $\Vert 1 - t x \Vert \leq 1$ and $\Vert 1-x \Vert = 1$ then every functional that achieves  its norm at 
$1-x$ is a scalar multiple of a state.  Hence if $x$ is an  element  of $A$ with
 $0 \notin W(x)$ and with $\Vert 1 - t x \Vert \leq 1$ for some $t > 1$ then
$\Vert 1 - x \Vert < 1$.
\end{proposition}   \begin{proof}  Any norm 1 functional $f$ with $f(1-x) = 1$, satisfies 
$$1 \geq |f(1-tx) |=  |f(1-t + t(1-x))| = |t - (t-1)f(1)| \geq t - (t-1)|f(1)| \geq t- (t-1) = 1.$$
So these are all equalities.   
It is clear that $f(1) \neq 0$.  By the converse to the triangle inequality, the second last and the last
(in)equality implies that $f(1) \geq 0$ and then that $f(1) = 1$.   So $f$ is a state.   

For the second assertion, by convexity $\Vert 1 - x \Vert \leq 1$.
 if $\Vert 1 - x \Vert =  1$ then 
by the first assertion there is a state that achieves its norm at 
$1-x$, so $f(x) = 0$ contradicting 
$0 \notin W(x)$.    
 \end{proof}

\section{The `sign' of a Banach algebra element}

In this section we point out that much of the theory of the
`sign of a matrix' summarized in \cite[Chapter 5]{High}  (this is sometimes called the `sector') generalizes
to Banach algebras or operator algebras.  We will follow the 
development in \cite[Chapter 5]{High} slavishly--our intent is
simply to repeat the results that generalize, and in each case
say a word about how the 
proof needs to be adapted if necessary.
  
By the spectral mapping theorem, if $x$ is an element of a unital Banach algebra with 
Sp$(x) \cap i \Rdb = \emptyset$, then
Sp$(x^2)$ contains no real negative numbers nor $0$.  So as we said
in the Introduction,
$x^2$ has a unique principal square root, whose inverse
we write as $(x^2)^{-\frac{1}{2}}$.  
We define $${\rm sign}(x) = x (x^2)^{-\frac{1}{2}} \; \; \; \;  
 \; \; \; {\rm if} \; {\rm Sp}(x) \cap i \Rdb = \emptyset.$$

As in the matrix theory, ${\rm sign}(x)$ has an integral formula
$${\rm sign}(x) = \frac{2x}{\pi} \int_0^\infty \, (t^2 1 + x^2)^{-1} \, dt.$$
This follows immediately from Lemma   \ref{Rieszroot}.

\begin{proposition} \label{sign}  Suppose that 
$a$ is an element of a unital Banach algebra $A$ with ${\rm Sp}(a) \cap i \Rdb = \emptyset$, and let $S = {\rm sign}(a)$.  \begin{itemize} \item [(1)]  $S^2 = 1$.
\item [(2)]  $S \in \{ a \}''$.
\item [(3)]  If $a$ is also a selfadjoint Hilbert space operator then 
$S$ is a symmetry (that is, a selfadjoint unitary).  More generally,
${\rm sign}(a^*) = {\rm sign}(a)^*$.
\item [(4)] $E_+ = \frac{1}{2} (I +S)$ and $E_{-} = \frac{1}{2} (I - S)$ are idempotents with sum $1$, and with
$S E_+ = E_+, S E_{-} = - E_{-},$ and $S = E_+ - E_-$.   Indeed $E_+$ is the spectral idempotent \cite{Dal}
of $a$ associated with ${\rm Sp}(a) \cap {\mathbb H}$.  
\item [(5)]  ${\rm Sp}(a) \subset {\mathbb H}$  iff $1 = {\rm sign}(a)$.   
\item [(6)]  ${\rm sign}(v^{-1} a v) = v^{-1} {\rm sign}(a) v$ if
$v$ is an invertible element of the algebra.
 \item [(7)]  $a = {\rm sign}(a) N$ where $N = (a^2)^{\frac{1}{2}}$.
\item [(8)]  ${\rm sign}(ca) = {\rm sign}(c) \, {\rm sign}(a)$  
if $c$ is a nonzero real scalar.  
\item [(9)] ${\rm sign}(a^{-1}) = {\rm sign}(a)$.
 \end{itemize}
\end{proposition} \begin{proof}  (1) and (7) are obvious, and (2) 
is clear since the square root is in $\{ a \}''$.  For  (3) use the fact that 
$*$ `commutes' with the inverse, and with the square root (we leave the latter as
a simple exercise using the uniqueness of the primary square root).  
The first assertions in (4) follow
from (1).  The `spectral idempotent' assertion is because working with respect to the Banach algebra
generated by $1$ and $a$, if $\chi$ is a character of $A$ with $\chi(a) \in {\mathbb H}$ then 
$\chi(a) \cdot (\chi(a)^2)^{-\frac{1}{2}} = 1$.  And if $\chi$ is a character with $\chi(a) \in -{\mathbb H}$
then $\chi(a) \cdot (\chi(a)^2)^{-\frac{1}{2}} = -1$.

Since  ${\rm Sp}(a) \subset {\mathbb H}$ iff $(a^2)^{\frac{1}{2}} = a$,
item (5) is clear.    
For (6), $${\rm sign}(v^{-1} a v) = (v^{-1} a v) (v^{-1} a^2 v)^{-\frac{1}{2}}
= v^{-1} {\rm sign}(a) v.$$
We are silently using the uniqueness property of the principal square root
here.  We leave (8) as an exercise, and
(9) is simple algebra using the relations $a \cdot a = 
(a^2)^{\frac{1}{2}} \cdot (a^2)^{\frac{1}{2}}$ and
$((a^2)^{\frac{1}{2}})^{-1} = ((a^2)^{-1})^{\frac{1}{2}}
= ((a^{-1})^{2})^{\frac{1}{2}}$.   One may also deduce (9) from Theorem
\ref{Newtsign} below. 
  \end{proof}  

\begin{proposition} \label{signmat}  For operators $a, b$ on a Hilbert space
such that ${\rm Sp}(ba)$ contains no negative real numbers nor zero,  
we have $${\rm sign} \left[ \begin{array}{ccl} 0 & a \\
b & 0 \end{array} \right] = \left[ \begin{array}{ccl} 0 &  c \\ c^{-1} & 0 \end{array} \right] $$
where $c = a (ba)^{-\frac{1}{2}}$.   \end{proposition} \begin{proof} 
Since it is well known that 
${\rm Sp}(ab) \setminus \{ 0 \} = {\rm Sp}(ba) \setminus \{ 0 \}$, 
we also have that ${\rm Sp}(ab)$ contains no negative real numbers nor zero.  
Using graduate level  operator theory it is clear that the rest 
of the proof of   \cite[Theorem 5.2]{High} works in infinite
dimensions.  \end{proof}

\begin{remark}   1)\  It is
clear that Proposition \ref{signmat} works for Banach algebras 
too for any appropriate norm on $M_2(A)$.

\medskip

2) \ It is no doubt true as in the matrix case 
 that ${\rm sign}(a) = \frac{2}{\pi} \lim_{t \to \infty} \, \arctan (ta)$ 
for any  element $a$ 
of a unital Banach algebra $A$ with ${\rm Sp}(a) \cap i \Rdb = \emptyset$.   Indeed this
boils down to showing that $\int_0^t \, (s^2 1 + a^2)^{-1} \, ds = \arctan (ta)$ for 
positive scalars $t$, and the latter is possibly well known.     
 \end{remark} 

It follows that for an invertible operator $a$ on a Hilbert space
with no negative numbers in its spectrum,  we have 
$${\rm sign} \left[ \begin{array}{ccl} 0 & a \\
I & 0 \end{array} \right] = \left[ \begin{array}{ccl} 0 &  a^{\frac{1}{2}} \\
a^{-\frac{1}{2}} & 0 \end{array} \right] .$$

We now turn to the (iterative ) Newton method $X_{k+1} = \frac{1}{2}(X_k + X_k^{-1})$
for ${\rm sign}(a)$.  We will take $X_0 = a$.

\begin{theorem} \label{Newtsign}  Suppose that
$a$ is an element of a unital Banach algebra with ${\rm Sp}(a) \cap i \Rdb = \emptyset$, and let $S = {\rm sign}(a)$. 
Then  the  Newton iterates $X_k$ above for ${\rm sign}(a)$ 
 converge quadratically to $S$, and also $X_k^{-1} \to S$, with
$$\Vert X_{k+1} - S \Vert \leq \frac{1}{2} \Vert X_k^{-1} \Vert
\Vert X_{k} - S \Vert^2 ,$$ and $X_k = (1 - G_0^{2^k})^{-1} 
(1+G_0^{2^k}) S$ for $k \geq 1$, where $G_0 = \kappa(N)$, where $N = 
(a^2)^{\frac{1}{2}}$.   
\end{theorem} \begin{proof}   We adjust the proof in \cite[Theorem 5.2]{High} slightly, and omit several easy 
details.
By the spectral mapping theorem, since the spectrum of $N$  
 lies in the open right half plane, the spectrum of $G_0$ lies
 in the open unit ball, and hence also the spectrum of $G_0^{2^k}$ lies in this ball.   
So $(1 - G_0^{2^k})^{-1}$ exists.   Set $X_k = (1 - G_0^{2^k})^{-1} 
(1+G_0^{2^k}) S$; we will show that $X_{k+1} = \frac{1}{2}(X_k + X_k^{-1})$.   Indeed
 $\frac{1}{2}(X_k + X_k^{-1}) =  \frac{S}{2}( (1 - G_0^{2^k})^{-1} 
(1+G_0^{2^k})  +  (1 - G_0^{2^k})
(1+G_0^{2^k})^{-1} )$ equals
$$\frac{S}{2} (1 - G_0^{2^k})^{-1} 
(1+G_0^{2^k})^{-1} [(1 - G_0^{2^k})^{2}+ 
(1+G_0^{2^k})^2] = \frac{S}{2} (1 - G_0^{2^{k+1}})^{-1} [2(1 + G_0^{2^{k+1}})]$$
which equals $X_{k+1}$.
Since the spectral radius of $G_0$ is smaller than $1$, 
it follows that $G_0^{2^k} \to 0$ as $k \to \infty$, so that 
$X_k = (1 - G_0^{2^k})^{-1} 
(1+G_0^{2^k}) S \to S$ (we are using the continuity of the inverse at $1$ in a Banach algebra).
Similarly, $X_k^{-1} = (1 - G_0^{2^k}) 
(1+G_0^{2^k})^{-1} S \to S$.   The rest is as in  \cite[Theorem 5.6]{High}.
\end{proof}

\begin{remark}  A common application of the sign function for matrices in numerical analysis 
and engineering is to  solve $ax - xb = y$ for $x$.  Suppose  that the spectrum of $a$ is in the 
negative right half plane and the spectrum of $b$ is in the positive
right half plane.   As on    \cite[p.\ 11]{BhRo}, we have
$$ \left[ \begin{array}{ccl} a & y \\
0 & b \end{array} \right] = \left[ \begin{array}{ccl} 1 & -x \\
0 & 1 \end{array} \right] \, \left[ \begin{array}{ccl} a & 0 \\
0 & b \end{array} \right]  \, \left[ \begin{array}{ccl} 1 & x \\
0 & 1 \end{array} \right] .$$    
The sign of the matrix in the middle is the diagonal  matrix with diagonal  entries $1$ and $-1$, and so it 
follows from Proposition \ref{sign} (6) that 
$$ {\rm sign} \Bigl( \left[ \begin{array}{ccl} a & y \\
0 & b \end{array} \right] \Bigr) = \left[ \begin{array}{ccl} 1 & -x \\
0 & 1 \end{array} \right] \, \left[ \begin{array}{ccl} 1 & 0 \\
0 & - 1 \end{array} \right]  \, \left[ \begin{array}{ccl} 1 & x \\
0 & 1 \end{array} \right] = \left[ \begin{array}{ccl} 1 & 2x \\
0 & -1 \end{array} \right] .$$
Thus $x$ is one half of the $1$-$2$ entry of ${\rm sign} \Bigl( \left[ \begin{array}{ccl} a & y \\
0 & b \end{array} \right] \Bigr)$.  
\end{remark} 

\section{Newton's method for the square root}

Newton's method for the  square root $a^{\frac{1}{2}}$ is
$$X_{k+1} = \frac{1}{2} \, (X_k  + X_k^{-1} \; a),$$   
with $X_0 = I$ usually.

Define $\kappa(\lambda) = \frac{\lambda - 1}{\lambda + 1}$ for $\lambda \in \Cdb, \lambda \neq -1$.  
This map takes the right half plane onto the unit circle (omitting the number $1$).  The inverse of this is the map
$\kappa^{-1}(\lambda) = \frac{1 + \lambda}{1 - \lambda}$.  (Some authors use instead the map 
$\lambda \mapsto \frac{1 - \lambda}{1 + \lambda},$ which is its own inverse.)

\begin{lemma} \label{supq}  Fix $n \in \Ndb$.
The supremum of $\frac{t \, \kappa(t)^{2^n}}{1 - \kappa(t)^{2^n}}$  on $(0,1]$ is  $\frac{1}{2^{n+1}}$, which it converges to as $t \to 0$.    
\end{lemma} \begin{proof} To see this, let us change
variables, letting $s = - \kappa(t)$, so that $t = -\kappa(s)$.  Then the function to be 
maximized is $\frac{|\kappa(s)| \, s^k}{1-s^k}$, for $s \in [0,1)$ and
$k = 2^n$.   We claim that this is an increasing function.  Indeed if one takes its derivative,
the denominator is positive as usual, and  the numerator on $(0,1)$ is a positive multiple of
$(-2s + k (1-s^2))(1-s^k) + k(1-s^2)t^k$, and the latter equals
$$k(1-s^2) - 2s (1-s^k) \geq (1-s) [k (1+s) -2ks] = k(1-s)^2 \geq 0,$$
since $2ks(1-s) \leq 2s (1-s^k)$.       Thus the function is increasing, and its supremum
is its limit as $t \to 1^{-}$, which by L'Hopitals rule is $\frac{1}{2^{n+1}}$.
\end{proof}

We now turn to the square root, which has
many equivalent definitions (see e.g.\ \cite[Section 6]{BSan}).  For example it has formula
$$x^{\frac{1}{2}} = \frac{2}{\pi} \, x \, \int_0^\infty \, (t^2 1 + x)^{-1} \, dt,$$
 if $x$ is type $M$ (by substituting $u = t^{\frac{1}{2}}$ 
in the Balakrishnan formula (3.2) in \cite{Haase}), or if $x$ is invertible and 
the spectrum  of $x$ contains no real strictly negative numbers (by Lemma   \ref{Rieszroot}).   

\begin{theorem} \label{Newtsqi}  Suppose that
$a$ is an element of a unital Banach algebra $A$ with ${\rm Sp}(a)$ containing no negative
real numbers nor $0$.  Suppose that $X_0 \in \{ a \}'$ with ${\rm Sp}(a^{-\frac{1}{2}} X_0)$
contained   in the open right half plane.        Then the  Newton iterates $X_k$ above for the square root 
 converge quadratically to $a^{\frac{1}{2}}$, and also $X_k^{-1} \to a^{-\frac{1}{2}}$, with
$$\Vert X_{k+1} - a^{\frac{1}{2}} \Vert \leq \frac{1}{2} \Vert X_k^{-1} \Vert
\Vert X_{k} - a^{\frac{1}{2}} \Vert^2 ,$$ and $X_k = a^{\frac{1}{2}} (1 - G_0^{2^k})^{-1} 
(1+G_0^{2^k}) S$ for $k \geq 1$, where $G_0 = \kappa(N)$, where $N = 
((a^{-\frac{1}{2}} X_0)^2)^{\frac{1}{2}}$.   
\end{theorem} \begin{proof}   The proof in \cite[Theorem 6.9]{High} works in our setting
too, using our Theorem \ref{Newtsign}
in place of  \cite[Theorem 5.2]{High}.  
\end{proof}  

\begin{remark}   We point out  that  if $A$ is an operator algebra then in the situation of Theorem \ref{Newtsqi} we
also get that if $X_0$ and $X_0^{-2} a$ are  accretive, then 
$X_k$ and $X_k^{-2} a$ are accretive, and  
$X_k^{-1} a^{\frac{1}{2}}$ has numerical range in $S_{\frac{\pi}{4}}$, for all $k$.      We prove this by induction.
If it is true for $k$ then 
$$X_{k+1}^2 a^{-1} = \frac{1}{4}(X_k^2 a^{-1} + 2 \cdot 1 + X_k^{-2}a).$$
All three parts of this are accretive, using Lemma  \ref{inv}.  So $X_{k+1}^2 a^{-1}$ is
accretive, and so also is $X_{k+1}^{-2} a$ by Lemma  \ref{inv}.   Also, $X_{k+1}^{-1} a^{\frac{1}{2}}$
has spectrum in the right half plane as we shall see soon
(around Equation (\ref{two}) below), so  $X_{k+1}^{-1} a^{\frac{1}{2}}$
is the principal square root of $X_{k+1}^2 a^{-1}$ and has numerical range in $S_{\frac{\pi}{4}}$.
Then $X_{k+1}  = \frac{1}{2}(X_k + (X_k^{-1} a^{\frac{1}{2}}) a^{\frac{1}{2}}).$
Now the product of two commuting operators with numerical range in $S_{\frac{\pi}{4}}$ is
accretive \cite{BBS}.  Hence $X_{k+1}$ is accretive, being the average of two accretives.
\end{remark}

We next discuss Newton's method for noninvertible $a$.   
This works for a rather large class of elements in operator algebras.  We will usually take 
$X_0 = 1$ or $X_0 = (a+1)/2$
(note that if $X_0 = 1$ then $X_1 = (a+1)/2$, so we may as well assume $X_0 = 1$).  
An important remark about other `starting values' is stated just after Proposition \ref{newpba}.

\begin{theorem} \label{preq}  If $a$ is an
operator on a Hilbert space 
 with numerical range $W(a) \subset S_\theta$ for some $\theta < \pi$, 
then Newton's method for the square root, with $X_0 = 1$ or $X_0 = (a+1)/2$,
converges to the principal  square root $a^{\frac{1}{2}}$.   Indeed for $n$ large enough,
the $n$th iterate $X_n$  
in Newton's method has distance less than $\frac{C_{\rho} K}{2^{n}}$ 
from $a^{\frac{1}{2}}$.  Here $K$ is Crouziex's constant
(mentioned at the end of the introduction),
and $C_{\rho}$ is any constant greater than $\sec(\frac{\rho}{2})$
where $\rho$ is  the sectorial 
angle of $a$ (thus $W(a) \subset S_\rho$).  In particular,
if $a$ is accretive then $\Vert X_n - a^{\frac{1}{2}} \Vert 
\leq \frac{K}{2^{n-1}}$ for all $n$ large enough.
\end{theorem} 

\begin{proof}   
First we work in any unital Banach algebra.   
Let $c = a^{\frac{1}{2}}$, whose spectrum
 is contained in a sector  $S_{\theta}$ where 
$\theta < \frac{\pi}{2}$ (see Theorem {\rm \ref{lrs}}).    For now let
$X_0$ be any invertible in the algebra with the property
that $d = X_0^{-1} c$ 
satisfies that  Sp$(d) \setminus \{ 0 \}$ 
is in the open right half plane (this is clearly 
true if $X_0 = 1$ (and we will see that it is
true if $X_0 = (a+1)/2$ and hence also  if $X_0 = a+1$)).  Let $G_0 = (1-d)(1+
d)^{-1}$.
This is the negative of the Cayley transform $\kappa(d)$ of $d$.
  We note that $1$ is in the spectrum of $G_0$ if $c$ is not invertible.
However $-1$ is never in the spectrum of $G_0$.    Indeed  the following is true:

Claim: $1$ is the only number in the spectrum of $G_0$ which has modulus $1$.   
The elements in the spectrum of  $G_0$ with modulus $1$ correspond, by the spectral mapping theorem,
 to elements in the spectrum of  $\kappa(d)$ with modulus $1$,
and these correspond to 
purely imaginary elements (or $0$) in the spectrum of  $d$.
By our hypothesis on $d$ above only $0$ is possible.     However the latter $0$ 
would correspond to 
$1$ in the spectrum of $G_0$, not to $-1$.   

From the Claim it follows also that  $G_k = G_0^{2^k}$  does not 
have $-1$ in its spectrum.  

We next claim that $X_n$ is invertible
and in fact \begin{equation} \label{one}
X_n = \frac{(X_0+c)}{2} (1+G_0^{2^n})  [(1+G_0) (1+G_0^2) \cdots (1+G_0^{2^{n-1}})]^{-1} , \qquad n \in \Ndb. \end{equation}
 We prove (\ref{one}) by induction.     We leave it to the reader to check the case $n = 1$.  Assume it is true for $n$.
We use the polynomial identity 
\begin{equation} \label{polyid}   (1-z) \prod_{k=0}^{n-1} (1+z^{2^k}) = 1 - z^{2^n} \end{equation} setting $z = G_0$.
Note that $1 - G_0 = 2c (X_0+c)^{-1}$,  so that $2c (X_0+c)^{-1} [(1+G_0) (1+G_0^2) \cdots (1+G_0^{2^{n-1}})] = 1 - G_0^{2^n}$.
Now $X_n^{-1} c$ equals
$$2c (X_0+c)^{-1}  (1+G_0^{2^n})^{-1} 
 [(1+G_0) (1+G_0^2) \cdots (1+G_0^{2^{n-1}})]  
= (1+G_0^{2^n})^{-1} (1-G_0^{2^n}).$$
That is,
\begin{equation} \label{two}  X_n^{-1} c = -\kappa(G_0^{2^n}). \end{equation}
By the spectral mapping theorem and what we said earlier about elements in the spectrum of  $G_0$ with modulus $1$,
it  follows that Sp$(X_n^{-1} c) \setminus \{ 0 \}$ is contained in the open right half plane.
We remark in passing that in the Hilbert space
operator case and $X_0^{-1} c$ is accretive (which is true if e.g.\ $X_0 = 1$),  then 
by  the theory of the Cayley transform $G_0$ is a contraction, hence  $\Vert - G_0^{2^n} \Vert \leq 1$, and 
so  $X_n^{-1} c$ is accretive.
As we saw earlier, if $a$ is an invertible 
operator on a Hilbert space and $X_0 = 1$ then $W(X_n^{-1} c) \subset S_{\frac{\pi}{4}}$ for all $n$.   (We imagine that 
this should be true even if $a$ is not   invertible.)  

Thus $$X_{n+1} = \frac{1}{2} (X_n + X_n^{-1}a) = \frac{X_n}{2} (1 + (X_n^{-1} c )^2)
= \frac{X_n}{2} (1 + \kappa(G_0^{2^n})^2),$$ 
which equals $\frac{X_n}{2} (2 (1 + G_0^{2^{n+1}})(1 + G_0^{2^n})^{-2}))$,
 using the easily checked  identity $1 + \kappa(w)^2 = 2(1 + w^2)(1+w)^{-2}$, which is
true for any $w$ with $1 + w$ invertible.  
Thus $$X_{n+1} =  \frac{(X_0+c)}{2} (1+G_0^{2^{n+1}})  [(1+G_0) (1+G_0^2) \cdots (1+G_0^{2^{n}})]^{-1}$$  as desired in the 
induction step.  

Suppose that $X_0 = p(c)$ where $p(z)$ is a nonvanishing analytic function
on a neighborhood of the spectrum of $c$.  Our assumption on $d$ above
follows if $q(z) = z/p(z)$ is in the open right half plane for all
$z \in {\rm Sp}(c) \setminus \{ 0 \}$.  This in turn follows for example if $a$ 
is accretive (so that $W(c) \subset S_{\frac{\pi}{4}}$) and if $p({\rm Sp}(c)) \subset S_{\frac{\pi}{4}}$.  
We thus have $X_n - c = f_n(c)$ where $$f_n(z) = \frac{(p(z)+z)}{2} (1+(\kappa \circ q)^{2^{n}}) \, [(1+\kappa(q(z))) (1+\kappa(q(z))^{2}) \cdots (1+\kappa(q(z))^{2^{n-1}})]^{-1} - z.$$
This is a rational function.      Indeed  using the polynomial identity (\ref{polyid})
 we have
\begin{equation} \label{three} f_n(z) = \frac{z(1+\kappa(q(z))^{2^{n}})}{1 - \kappa(q(z))^{2^{n}}} - z
=  \frac{2z \, \kappa(q(z))^{2^{n}}}{1 - \kappa(q(z))^{2^{n}}}, \qquad {\rm Re} \, z >  0, \end{equation}
and $f_n(0) = \frac{p(0)}{2^{n}}$.  (We note that assuming that $q(z) = z/p(z)$ is in the open right half plane for all
$z \in {\rm Sp}(c) \setminus \{ 0 \}$, forces 
$|\kappa(q(z))| = 1$ only when $q(z) = 0$, that is, only when $z = 0$.
The question is whether $f_n(c) \to 0$ as $n \to \infty$.    This would follow from the
continuity of the functional calculus {\em if} all of the $f_n$ were analytic on a fixed neighborhood of $0$, but unfortunately that is not generally the case.)

We remark that if $X_0 = 1$
then $G_0 = (1-c)(1+c)^{-1}$, the negative of the 
 Cayley transform $\kappa(c)$ of $c$.  
Equation (\ref{one}) becomes
\begin{equation} \label{one'}
X_n = \frac{1+c}{2} \; (1+G_0^{2^n}) (1+c)^2  [(1+G_0) (1+G_0^2) \cdots (1+G_0^{2^{n-1}})]^{-1} , \qquad n \in \Ndb. \end{equation} 
We still have $X_n - c = f_n(c)$, but the formula for $f_n$ 
in this case (c.f.\ the centered formula a few lines above Equation (\ref{three}))
  becomes
\begin{equation} \label{three''}
f_n(z) = \frac{1+z}{2} \; (1+\kappa(z)^{2^{n}}) [(1+\kappa(z)) (1+\kappa(z)^{2}) \cdots (1+\kappa(z)^{2^{n-1}})]^{-1} - z. \end{equation}
and so again using the polynomial identity (\ref{polyid})
 Equation (\ref{three}) becomes
 \begin{equation} \label{three'} f_n(z) = \frac{z(1+\kappa(z)^{2^{n}})}{1 - \kappa(z)^{2^{n}}} - z
=  \frac{2z \, \kappa(z)^{2^{n}}}{1 - \kappa(z)^{2^{n}}}, \qquad {\rm Re} \, z >  0,\end{equation}
and $f_n(0) = \frac{1}{2^{n}}$.   Again, the question is whether $f_n(c) \to 0$ as $n \to \infty$, which  would follow from the continuity of the Riesz
functional calculus {\em if} all of the $f_n$ were analytic on a fixed neighborhood of $0$, but unfortunately
that is not the case.  However if we are in an operator algebra then one may use 
a variant of the functional calculus for spectral sets, for example Crouzeix's 
analytic functional calculus mentioned at the end of the introduction.  This we now do. 

Henceforth,
assume we are in an operator algebra, and that $X_0 = 1$.
The numerical range
$W(c)$ is contained in $S_{\theta}$ where $\theta = \rho/2$ (see e.g.\ \cite{LRS,BBS}).
We will assume for clarity that $a$ is accretive, and so we may 
take $\theta = \frac{\pi}{4}$, the case that $\frac{\pi}{4} < \theta < \frac{\pi}{2}$
will be discussed at the end.  
It is a well known result of Crouzeix 
that the  numerical range of any operator is a $K$-{\em spectral set}
 for a positive constant 
$K < 12$.   Thus $\Vert  f_n(c) \Vert \leq K \Vert f_n \Vert_{W(c)}$ for a constant $K$
depending on the shape of (a closed region containing) the numerical range
of $c$.    
We will estimate  
$\Vert f_n \Vert_{E}$ where $E$ is the sector
of the circle of radius $\Vert c \Vert$ contained in $S_{\theta}$,
and hence see that $\Vert f_n \Vert_{W(c)} \leq \Vert f_n \Vert_{E}  \to 0$ as $n \to \infty$.
It is easy to see that $f_n$ has limit $\frac{1}{2^{n+1}}$ as one
approaches $0$ from the right.
Fix a small $\delta > 0$.  If one considers the picture of the image  of $E$ 
under the map $z \mapsto 
\frac{1-z}{1+z}$, one sees that $|\kappa(z)| < 1- \delta$ for all $z \in E
\setminus D(0,\epsilon)$, for a small $\epsilon > 0$ (independent
of $n$).
Hence for such $z$ we have
$$|f_n(z)| \leq  \frac{2|z| \, |\kappa(z)|^{2^{n}}}{1 - |\kappa(z)|^{2^{n}}}
\leq \frac{2 \Vert c \Vert (1- \delta)^{2^{n}}}{1 - (1- \delta)^{2^{n}}}$$  
The right side will be less than $\frac{1}{2^{n}}$ for $n$ large enough,
and so we see that for $n$ large enough, 
the maximum of $|f_n(z)|$ is achieved on $E \cap D(0,\epsilon)$.
By a similar argument,
and the maximum modulus theorem, the maximum of $|f_n(z)|$ is achieved on 
the boundary lines of $E \cap D(0,\epsilon)$, and by symmetry on the 
upper of these two lines.  Thus if $\theta = \frac{\pi}{4}$
we may assume that $z = t(1+i)$ for
$0 < t < \epsilon$.  Using the identity
$$|\kappa(z)|^{2} =  \frac{1-2 {\rm Re} z + |z|^2}{1+2 {\rm Re} z + |z|^2}
= -\kappa(\frac{2 {\rm Re} \, z}{1 + |z|^2}),$$ we see that  
\begin{equation} \label{fi} |f_n(z)| \leq \frac{2|z| \, |\kappa(z)|^{2^{n}}}{1 - |\kappa(z)|^{2^{n}}}
=
\frac{2 \sqrt{2} t \kappa(s)^{2^{n-1}}}{1 - \kappa(s)^{2^{n-1}}}
= \frac{\sqrt{2} s \kappa(s)^{2^{n-1}}}{1 - \kappa(s)^{2^{n-1}}} (1+2t^2), \end{equation}
where $s = \frac{2t}{1+2t^2}$. 
By Lemma \ref{supq} the supremum of  the last function is $\leq \frac{\sqrt{2}}{2^{n}} (1+2t^2)
< \frac{1}{2^{n-1}}$ if $2 \epsilon^2 < \sqrt{2}-1$.  
Thus given $\epsilon > 0$ we see that
for $n$ large enough we  have  
$\Vert f_n \Vert_{E} \leq \frac{\sqrt{2}}{2^{n}} (1+\epsilon)
< \frac{1}{2^{n-1}}$.  Hence $$\Vert X_n - a^{\frac{1}{2}} \Vert  \leq  \frac{K}{2^{n-1}}.$$ 
(Note that if we do not assume
$X_0 = (1 +a)/2$, but instead $X_0 = p(c)$ as we had earlier, then 
the same analysis shows that
$$|f_n(z)| \leq  \frac{2 |z| \kappa(t)^{2^{n-1}}}{1 - \kappa(t)^{2^{n-1}}}$$
where now $t = \frac{2 \, {\rm Re} \, q(z)}{1 + |q(z)|^2}$, which is still
in $[0,1]$.  However it may not be easy to dominate
$|z|$ by a multiple of this $t$ as we did before, unless $p(z)$ is of a 
very special form, like $(1+z)/2$).)

If $W(c) \subset S_\theta$ for $\theta < \pi/2$, set $z = t e^{i \theta}$,
and Equation (\ref{fi}) becomes 
$$|f_n(z)| \leq \frac{2t \kappa(s)^{2^{n-1}}}{1 - \kappa(s)^{2^{n-1}}} 
\leq \sec (\theta) \frac{s \kappa(s)^{2^{n-1}}}{1 - \kappa(s)^{2^{n-1}}} (1+t^2) 
\leq \frac{\sec (\theta) \, (1+t^2)}{2^{n}} < \frac{C_{\rho}}{2^{n}},$$ 
for $\epsilon$ small enough, where $s = \frac{2t \cos \theta}{1+t^2}$
and $C_{\rho}$ is any constant greater than $\sec(\frac{\rho}{2})$.
    \end{proof}  
  
\begin{remark}    1) \ With a little more work in the last proof one should be able to
show that 
the maximum of $|f_n(z)|$ on 
$W(c)$, or on the intersection $S_{\theta}$ with the disk of radius $\Vert c \Vert$, 
is achieved at $0$.    This also seemed to be confirmed by numerical 
computations for various values of $n$.   If this is the case
 then $C_\rho$ may be replaced by $1$ in the 
estimate in the last result.   That is,
$\Vert X_n - a^{\frac{1}{2}} \Vert  \leq  \frac{K}{2^{n}}$.

\medskip 

2) \  Thus $\Vert X_n \Vert \leq \Vert a^{\frac{1}{2}} \Vert + \frac{C}{2^{n-1}}$ 
for a constant $C$.
One should also be able to get an estimate for $\Vert X_n^{-1} \Vert$.
Indeed using Crouziex's functional calculus
$\Vert X_n^{-1} \Vert \leq K \Vert g_n \Vert_{W(c)}$ where
$$g_n(z) = \frac{2}{1+z} \, 
(1+\kappa(z)^{2^{n}})^{-1} [(1+\kappa(z)) (1+\kappa(z)^{2}) \cdots (1+\kappa(z)^{2^{n-1}})].$$ 
We expect that $\Vert g_n \Vert_{W(c)} = 2^{n}$ if $a$ is not invertible (indeed
in this case we have $\Vert g_n \Vert_{W(c)} \geq g_n(0) = 2^{n}$).

\medskip

3)\  If $a$ is accretive one may apply Newton's method
to $a + \frac{1}{n} 1$, to get approximants for $a^{\frac{1}{2}}$.
This suggests at first sight 
that the following variant of Newton's method
might work: $X_{n+1} = \frac{1}{2}  (X_n + X_n^{-1} (a + \frac{1}{n} 1))$.
However since $X_n^{-1}$ may be growing at an order of $2^n$ or faster
this seems dangerous.  We conjecture that 
$X_{n+1} = \frac{1}{2}  (X_n + X_n^{-1} (a + \frac{1}{3^n}1))$ would 
work for all accretive operators $a$ on a Hilbert space, and possibly also in a
Banach algebra.
\end{remark}

\begin{proposition} \label{preqmat}  
If $x$ is a matrix with no strictly 
negative eigenvalues, and a square root in $\{ x \}''$,
 then Newton's method with $x_0 = (x+1)/2$ converges to 
 the principal  square root $x^{\frac{1}{2}}$. \end{proposition} \begin{proof}    If $x$ is invertible then this follows from
\cite[Theorem 6.9]{High}.  By \cite[Theorem 6.10]{High} we just need to show that if $0$ 
is an eigenvalue
then it is a semisimple eigenvalue, that is, there is no nontrivial Jordan block
for the eigenvalue $0$.   If there was such a nontrivial Jordan block $J_0$ then
first suppose that $x = V^{-1} J_0 V$.  Then $x$ has no
square root  as is well known
(see  e.g.\ \cite[Exercise 1.25]{High}).   Otherwise, suppose that $x = V^{-1} (J_0 
\oplus z) V$.   if $p$ is the support projection of $J_0$ then
$V^{-1} p V$ commutes with $x$ and hence also with $x^{\frac{1}{2}}$.
Thus $V^{-1} p V x^{\frac{1}{2}}$ is a square root of
$V^{-1} p V x =  V^{-1} J_0 V$.  However $J_0$ has no square root as we said above,
a contradiction.  \end{proof}

\section{The geometric mean, and solving $x a^{-1} x = b$}  
 In this section we note that Drury's 
results from \cite[Section 3]{Dr} for the geometric mean of matrices with
(strictly) positive definite real part, generalize to strictly accretive elements in
a unital operator algebra.   We also establish a few more aspects of this mean.
We remark that the geometric mean of positive matrices and operators dates back 
to work of Pusz-Woronowicz and Ando (see \cite{LL} for a survey).  

\begin{theorem} \label{straccr}  Let $a$ and $b$ be strictly accretive elements in
a unital operator algebra.  Then 
$$G = a^{\frac{1}{2}} (a^{-\frac{1}{2}} b a^{-\frac{1}{2}})^{\frac{1}{2}} a^{\frac{1}{2}}$$
is strictly accretive too.  Moreover $G$ is the unique strictly accretive solution
to the equation $x a^{-1} x = b$,
and $G = b^{\frac{1}{2}} (b^{-\frac{1}{2}} a b^{-\frac{1}{2}})^{\frac{1}{2}} b^{\frac{1}{2}}.$ 
\end{theorem}

\begin{proof}  We slightly rewrite Drury's argument.   The first part of the
proof works in any unital Banach algebra: 
note that $$t 1 + a^{-\frac{1}{2}} b a^{-\frac{1}{2}} = a^{-\frac{1}{2}} (t a + b)a^{-\frac{1}{2}}, \qquad t \geq 0,$$
is invertible since $ta + b$ is strictly accretive.  So the spectrum of
$a^{-\frac{1}{2}} b a^{-\frac{1}{2}}$ contains no negative numbers or $0$,
and by the spectral mapping theorem the spectrum of
$(a^{-\frac{1}{2}} b a^{-\frac{1}{2}})^{\frac{1}{2}}$ is contained
in ${\mathbb H}$.  Similarly for the spectrum of
$(b^{-\frac{1}{2}} a b^{-\frac{1}{2}})^{\frac{1}{2}}$.
Clearly $G a^{-1} G = b$. 
By Lemma \ref{Rieszroot} we have
$$\frac{2}{\pi} \, \int_0^\infty \, a^{-\frac{1}{2}} \, (t^2 1 +
a^{-\frac{1}{2}} b a^{-\frac{1}{2}})^{-1} a^{-\frac{1}{2}} \, dt =
a^{-\frac{1}{2}} (a^{-\frac{1}{2}} b a^{-\frac{1}{2}})^{-\frac{1}{2}}
a^{-\frac{1}{2}} = G^{-1}.$$
We may rewrite this (convergent) integral in the more symmetric form  
$$\frac{2}{\pi} \, \int_0^\infty \, (t a + \frac{1}{t} b)^{-1} \, \frac{dt}{t}.$$
At this point we assume that $A$ is an operator algebra.
Note that for $0 < t < \infty$ we have that 
$t a + \frac{1}{t} b$ is strictly accretive,
and so by Lemma \ref{inv}, so is $(t a + \frac{1}{t} b)^{-1}$.
By a basic fact about integrals of positive functions we see that
the integral yields a strictly accretive element.
By Lemma \ref{inv} the inverse $G$ is strictly accretive too.  
Making the substitution $u = 1/t$ in the integral, we see that
the symmetry is perfect, and so $G$ equals
$$a^{\frac{1}{2}} (a^{-\frac{1}{2}} b a^{-\frac{1}{2}})^{\frac{1}{2}} a^{\frac{1}{2}} =
b^{\frac{1}{2}} (b^{-\frac{1}{2}} a b^{-\frac{1}{2}})^{\frac{1}{2}} b^{\frac{1}{2}}.$$

The argument in \cite[Proposition 3.5]{Dr} shows that
there is a unique strictly accretive 
$G$ satisfying $Ga^{-1}G = b$.  There is one point in that proof where one needs
that the spectrum of $H G^{-1}$ contains no negative numbers, for $H$ as in that paper, 
but this follows since 
$$t 1 + H G^{-1} = G^{-1}(tG + H), \qquad t \geq 0,$$
is invertible since $t G + H$ is strictly accretive. 
\end{proof}

Drury writes $G$ in the last result as $a \# b$, the {\em geometric mean}. 
As in \cite[Proposition 3.1]{Dr} we deduce: 

\begin{corollary}[Drury] \label{numraco} 
If $a$ and $b$ are as in the last result, and if
$W(a)$ and $W(b)$ are inside $S_\theta$ for some $\theta < \frac{\pi}{2}$,
then $W(a \# b) \subset S_\theta$.    \end{corollary}

\begin{lemma} \label{istyp}   Let $a$ and $b$ be accretive operators
on a Hilbert space $H$
with $a$ strictly accretive.    Or, let $a$ and $b$ be accretive elements of a Banach algebra
with $a$ invertible and $b$ strictly accretive.  
 Then $a^{-\frac{1}{2}} b a^{-\frac{1}{2}}$ is
of type $M$.   
\end{lemma}   

\begin{proof}  If $a$ is strictly accretive  then there exists $\epsilon > 0$ with $a \geq \epsilon I$.
We have 
$$\Vert (a^{-\frac{1}{2}} b a^{-\frac{1}{2}} + t 1)^{-1} \Vert \leq 
\Vert a^{\frac{1}{2}} \Vert^2 \Vert (b + t a)^{-1} \Vert .$$
For $\zeta \in H$ we have
$$t \epsilon \Vert \zeta \Vert^2 \leq t \, {\rm Re} \langle a \zeta , \zeta \rangle 
\leq {\rm Re} \langle (b + t a) \zeta , \zeta \rangle \leq \Vert (b + t a) \zeta \Vert \,  \Vert   \zeta \Vert .$$
Dividing by $\Vert   \zeta \Vert$ and letting $\zeta = (b + t a)^{-1} \eta$ we obtain 
$$\Vert (b + t a)^{-1} \eta \Vert \leq \frac{1}{t \epsilon}  \, \Vert   \eta \Vert , \qquad \eta \in H.$$
It follows that 
$$\Vert (a^{-\frac{1}{2}} b a^{-\frac{1}{2}} + t 1)^{-1} \Vert \leq 
\Vert a^{\frac{1}{2}} \Vert^2  \frac{1}{t \epsilon}$$ so that 
 $a^{-\frac{1}{2}} b a^{-\frac{1}{2}}$ is
of type $M$.

If $a$ and $b$ are accretive elements of a Banach algebra
with $a$ invertible and $b$ strictly accretive, and if $t \geq 0$,  then 
$${\rm Sp}(ta+b) \subset W(ta+b) = \{  \varphi(ta+b) : \varphi\in
S(A)\}=\{t\varphi(a)+\varphi(b) : \varphi\in S(A)\}\subset \Hdb.$$
Thus  $ta+b$ is invertible, so that  $t 1 + a^{-\frac{1}{2}} b a^{-\frac{1}{2}}$
 is invertible and $-t \notin {\rm Sp}(a^{-\frac{1}{2}} b a^{-\frac{1}{2}})$.
Hence by the discussion  above Lemma  \ref{unsqr}, 
$a^{-\frac{1}{2}} b a^{-\frac{1}{2}}$ is type $M$. 
\end{proof}  

\begin{remark}   If  $a$ and $b$ are    strictly accretive it need not follow that 
 $W(a^{-\frac{1}{2}} b a^{-\frac{1}{2}})$ contains no negative numbers.   For example, 
let $a^{-1} = b$ be the $2 \times 2$ matrix with rows $[1 \; \, 1]$ and $[-2 \; \, \frac{1}{3}]$.
\end{remark} 

We define $a \# b = a^{\frac{1}{2}} (a^{-\frac{1}{2}} b a^{-\frac{1}{2}})^{\frac{1}{2}} a^{\frac{1}{2}}$
if $a$ is strictly accretive and $b$ is accretive, or
if $a$ invertible and accretive and  $b$ is strictly accretive.    If $a$ and $b$ are simply accretive, and $a$ is invertible, then we define $a  \# b = \lim_{\epsilon \to 0^+} \, a \# (b + \epsilon 1)$, as in the final assertion of the next result.   

\begin{corollary} \label{accr}  Let $a$ and $b$ be accretive elements in
a unital operator algebra with $a$ strictly accretive; or 
with $a$ invertible and $b$ strictly accretive.   Then 
$a \# b$
is  accretive too.    Indeed its numerical range is again  in  $S_\theta$ if 
$W(a)$ and $W(b)$ are inside $S_\theta$ for some $\theta \leq \frac{\pi}{2}$.   
This is also true if $a$ and $b$ are simply accretive, and $a$ is invertible, if we 
define $a  \# b = \lim_{\epsilon \to 0^+} \, a \# (b + \epsilon 1)$.    The latter 
limit exists and is accretive, and is a solution to the equation 
$z a^{-1} z = b$.  \end{corollary}

\begin{proof}  If $a$ is strictly accretive apply the theorem (or its proof) with $b$ replaced by $b + \epsilon$, and let $\epsilon \to 0^+$,
using Lemmas \ref{unsqr} and \ref{istyp}.    
This allows one to see that $$
\Vert (a^{-\frac{1}{2}} (b + \epsilon 1) a^{-\frac{1}{2}})^{\frac{1}{2}} - 
 (a^{-\frac{1}{2}} b a^{-\frac{1}{2}})^{\frac{1}{2}}
\Vert \leq K \Vert \epsilon a^{-1} \Vert^{\frac{1}{2}} \to 0$$ as $\epsilon \to 0$. 
So $a \# (b + \epsilon 1) \to a^{\frac{1}{2}} (a^{-\frac{1}{2}} b a^{-\frac{1}{2}})^{\frac{1}{2}} a^{\frac{1}{2}}$
as $\epsilon \to 0$.   Since $W(b+ \epsilon 1)  \subset S_\theta$ if 
$W(b) \subset S_\theta$ the last assertion follows easily from Corollary \ref{numraco}.

Suppose that $a$ is invertible and $b$ strictly accretive.    By the above  
$(a + \epsilon 1) \# b$ is accretive with numerical range  in  $S_\theta$ if 
$W(a)$ and $W(b)$ are inside $S_\theta$.   If $z_\epsilon = 
((a+ \epsilon 1)^{-\frac{1}{2}} b (a+ \epsilon 1)^{-\frac{1}{2}})^{\frac{1}{2}}$ we
have $(a + \epsilon 1) \# b - a^{\frac{1}{2}} (a^{-\frac{1}{2}} b a^{-\frac{1}{2}})^{\frac{1}{2}} a^{\frac{1}{2}}$
equal to $$((a + \epsilon 1)^{\frac{1}{2}} - a^{\frac{1}{2}}) \, z_\epsilon \, (a + \epsilon 1)^{\frac{1}{2}} 
+ a^{\frac{1}{2}} (z_\epsilon - z_0) (a + \epsilon 1)^{\frac{1}{2}} +
a^{\frac{1}{2}} z_0 \, ((a + \epsilon 1)^{\frac{1}{2}} - a^{\frac{1}{2}}). $$
The last of these three terms has limit $0$.   The middle term  has limit $0$ too since by  Lemmas \ref{unsqr} and \ref{istyp}, 
$$\Vert z_\epsilon - z_0 \Vert \leq K \Vert (a+ \epsilon 1)^{-\frac{1}{2}} b (a+ \epsilon 1)^{-\frac{1}{2}}
- a^{-\frac{1}{2}} b a^{-\frac{1}{2}} \Vert^{\frac{1}{2}} \to 0$$ as $\epsilon \to 0$.   This also uses the 
fact that $(a + \epsilon 1)^{\frac{1}{2}} \to a^{\frac{1}{2}}$, and the continuity of the inverse.
We deduce that $\Vert z_\epsilon \Vert$ is bounded independently of $\epsilon$,  and so
the first of the three terms has limit $0$ too. 
Thus $(a + \epsilon 1) \# b \to a^{\frac{1}{2}} (a^{-\frac{1}{2}} b a^{-\frac{1}{2}})^{\frac{1}{2}} a^{\frac{1}{2}}$
as $\epsilon \to 0$.  Hence the latter is accretive with numerical range  in  $S_\theta$, since the 
former has these properties.

Finally, if $a$ and $b$ are simply accretive, and $a$ is invertible, 
then  $a^{-\frac{1}{2}} (b + \epsilon 1) a^{-\frac{1}{2}}$ is type $M$ 
by Lemma \ref{istyp},  we have by Lemma \ref{unsqr} that 
$$\Vert  a \# (b + \epsilon_1 1) - a \# (b + \epsilon_2 1) \Vert
\leq K \Vert a^{\frac{1}{2}} \Vert^2 \Vert a^{-\frac{1}{2}}(\epsilon_1 - \epsilon_2)
a^{-\frac{1}{2}} \Vert^{\frac{1}{2}}.$$
Thus $(a \# (b + \epsilon) 1))$ is a Cauchy net, hence convergent to an accretive element $z$
such that, as above,   $W(z) \subset S_\theta$ if 
$W(a)$ and $W(b)$ are inside $S_\theta$.    It is an exercise that it is a solution to the equation 
$z a^{-1} z = b$. 
\end{proof}

\begin{remark}   In the setting of  Corollary \ref{accr}, the same proof and Corollary \ref{numraoffc} show that if in addition the arguments of numbers in $W(a)$ and $W(b)$ are inside
$[\alpha, \beta]$ for $-\frac{\pi}{2} \leq \alpha \leq \beta \leq  \frac{\pi}{2}$,  then the same is true for $W(a \# b)$.

We also remark that solutions to $z a^{-1} z = b$ are not unique without the earlier strictly accretive hypothesis.
(for example let $a = {\rm diag}(i,2), b = {\rm diag}(i,1/2)$).  However for any solution $z$ to this equation, 
$a^{-\frac{1}{2}} z a^{-\frac{1}{2}}$ is a square root of $a^{-\frac{1}{2}} b a^{-\frac{1}{2}}$.
\end{remark}

\begin{corollary} \label{accr2}  Let $a$ and $b$ be accretive invertible elements in
a unital operator algebra.
 Then with the definitions above, $a \# b = b \# a$. 
  \end{corollary}

\begin{proof}    First assume that $b$ is strictly accretive.
By the last proof,
we have
$$a^{\frac{1}{2}} (a^{-\frac{1}{2}} b a^{-\frac{1}{2}})^{\frac{1}{2}} a^{\frac{1}{2}} =
 \lim_{\epsilon \to 0^+} \, (a + \epsilon 1) \# b = \lim_{\epsilon \to 0^+} \, b \# (a + \epsilon 1)
=  b^{\frac{1}{2}} (b^{-\frac{1}{2}} a b^{-\frac{1}{2}})^{\frac{1}{2}} b^{\frac{1}{2}},$$
the last equality 
by a fact in the proof of Corollary \ref{accr}.

If $b$ is not strictly accretive, we have by the above that $a \# b$ equals 
$$\lim_{\epsilon \to 0^+} \, a^{\frac{1}{2}} (a^{-\frac{1}{2}} (b + \epsilon 1)
a^{-\frac{1}{2}})^{\frac{1}{2}} a^{\frac{1}{2}} 
=  \lim_{\epsilon \to 0^+} \, (b + \epsilon 1)^{\frac{1}{2}} ((b + \epsilon 1)^{-\frac{1}{2}}
a (b + \epsilon 1)^{-\frac{1}{2}})^{\frac{1}{2}} (b + \epsilon 1)^{\frac{1}{2}}.$$
This may be rewritten as $\lim_{\epsilon \to 0^+} \, b^{\frac{1}{2}} ((b + \epsilon 1)^{-\frac{1}{2}}
a (b + \epsilon 1)^{-\frac{1}{2}})^{\frac{1}{2}} b^{\frac{1}{2}}.$   Now both
$(b + \epsilon 1)^{-\frac{1}{2}}
a (b + \epsilon 1)^{-\frac{1}{2}}$ and $b^{-\frac{1}{2}}
(a+ \epsilon 1) b^{-\frac{1}{2}}$ are type $M$ by Lemma \ref{istyp}, 
so the norm of the difference of the square roots of these two products is dominated 
by Lemma   \ref{unsqr} by
$$ 
K \Vert (b + \epsilon 1)^{-\frac{1}{2}}
a (b + \epsilon 1)^{-\frac{1}{2}} - b^{-\frac{1}{2}}
(a+ \epsilon 1) b^{-\frac{1}{2}} \Vert^{\frac{1}{2}} \to 0$$
as $\epsilon \to 0$.   Thus we see
$$a \# b = \lim_{\epsilon \to 0^+} \, b^{\frac{1}{2}} (b^{-\frac{1}{2}}
(a+ \epsilon 1) b^{-\frac{1}{2}})^{\frac{1}{2}} b^{\frac{1}{2}} = b \# a$$
as desired.
\end{proof}   

 We remark that  in the case that
 both $a$ and $b$ are simply accretive and $a$ is invertible, $a^{\frac{1}{2}} (a^{-\frac{1}{2}} b a^{-\frac{1}{2}})^{\frac{1}{2}} a^{\frac{1}{2}}$ may not make sense.  For example if $b = a^{-1} = i$ in $\Cdb$ 
then $a^{-\frac{1}{2}} b a^{-\frac{1}{2}} = b^2 = -1$, whose principal root is undefined.

\begin{lemma} \label{cocas}  If $a$ and $b$ are accretive elements in
a unital operator algebra such that $a$ and $b$ commute, 
and if $a$ is 
invertible, 
then $a \# b
= a^{\frac{1}{2}} b^{\frac{1}{2}}$, and this is accretive.
\end{lemma} \begin{proof}   
First suppose that $b$ is also strictly accretive.
Claim: $a^{-\frac{1}{2}} b^{\frac{1}{2}}$ 
has spectrum in ${\mathbb H}$.   Indeed, if $\chi$ is a character
of the unital Banach algebra generated by $a$ and $b$ then
$\chi(a^{-\frac{1}{2}} b^{\frac{1}{2}}) = 
\chi(a)^{-\frac{1}{2}} \chi(b)^{\frac{1}{2}}$,
which is in ${\mathbb H}$.  
By the Claim and the unicity of roots mentioned in the Introduction,
we have $$(a^{-\frac{1}{2}} b a^{-\frac{1}{2}})^{\frac{1}{2}}
= (a^{-1} b)^{\frac{1}{2}} = a^{-\frac{1}{2}} b^{\frac{1}{2}}.$$
Hence $a^{\frac{1}{2}} (a^{-\frac{1}{2}} b a^{-\frac{1}{2}})^{\frac{1}{2}} a^{\frac{1}{2}}
= a^{\frac{1}{2}} b^{\frac{1}{2}}$.

More generally, $a \# b = \lim_{\epsilon \to 0^+} \, a \# (b+\epsilon 1) = \lim_{\epsilon \to 0^+} \, 
a^{\frac{1}{2}} (b+ \epsilon 1)^{\frac{1}{2}} = a^{\frac{1}{2}} b^{\frac{1}{2}}$.
\end{proof}

Similarly,  if $a$ and $b$ are any accretive elements in
a unital operator algebra such that $a$ and $b$ commute, 
then $ \lim_{\epsilon \to 0^+} \, (a+ \epsilon) \# b
= a^{\frac{1}{2}} b^{\frac{1}{2}}$, and this is accretive.  This shows that
 Corollary \ref{accr} is in some sense a noncommutative variant  
of the fact from \cite{BBS} that 
$a^{\frac{1}{2}} b^{\frac{1}{2}}$ is accretive for 
accretive commuting elements in a unital operator algebra.
We noted in \cite[Example 3.13]{BOZ} that the latter fact is  
false in a Banach algebra.   Hence none of the results 
above in this section are true for general Banach algebras.

It is easy to show that for positive scalars $s, t$ we have $(s a) \# (tb) = \sqrt{st} \, (a \# b)$
for $a, b$ accretive and $a$ invertible.   

\begin{proposition}   \label{invsg}   Let $a$ and $b$ be accretive invertible elements in
a unital operator algebra. 
 Then with the definitions above, $(a \# b)^{-1} = a^{-1} \, \# \, b^{-1}$.
\end{proposition} \begin{proof}  By Lemma \ref{inv}, $a^{-1}, b^{-1}$ and $(b + \epsilon 1)^{-1}$ are 
accretive.   Hence both $a^{\frac{1}{2}} (b + \epsilon 1)^{-1} a^{\frac{1}{2}}$ and 
$a^{\frac{1}{2}} (b^{-1}+ \epsilon 1) a^{\frac{1}{2}}$ are type $M$ by Lemma \ref{istyp}.
By Lemma   \ref{unsqr}   the norm of the difference of the square roots of these two products is dominated 
by
$$ 
K \Vert a^{\frac{1}{2}} (b + \epsilon 1)^{-1} a^{\frac{1}{2}} \, - \, 
a^{\frac{1}{2}} (b^{-1}+ \epsilon 1) a^{\frac{1}{2}} \Vert^{\frac{1}{2}} \to 0$$
as $\epsilon \to 0$.   Thus we see
$$a^{-1} \,  \# \, b^{-1} = \lim_{\epsilon \to 0^+} \, a^{-\frac{1}{2}} (a^{\frac{1}{2}}
(b+ \epsilon 1)^{-1} a^{\frac{1}{2}})^{\frac{1}{2}} a^{-\frac{1}{2}} .$$
It is easy to see that the product of  this and $a \# b = \lim_{\epsilon \to 0^+} \, a \# (b + \epsilon 1)$
is $$ \lim_{\epsilon \to 0^+} \, a^{-\frac{1}{2}} (a^{\frac{1}{2}}
(b+ \epsilon 1)^{-1} a^{\frac{1}{2}})^{\frac{1}{2}} a^{-\frac{1}{2}} \cdot
a^{\frac{1}{2}} (a^{-\frac{1}{2}}
(b+ \epsilon 1) a^{-\frac{1}{2}})^{\frac{1}{2}} a^{\frac{1}{2}} = 1$$
using the fact mentioned in the introduction that $(x^{-1})^{\frac{1}{2}} = (x^{\frac{1}{2}})^{-1}$.
Similarly, $(a \# b) (a^{-1} \# b^{-1}) = 1$
as desired. 
\end{proof}

\begin{proposition}    Suppose that $a,  c$ are invertible in $B(H)$, and that 
  $a$ and $b$ are accretive in $B(H)$. 
Then $c^{\ast}(a\#b) c=(c^*ac)\#(c^*bc)$.     In particular 
$(c^*bc)^{\frac{1}{2}}$ equals $c^{\ast}( (c c^*)^{-1} \# b) c$.
Also, $(a + b) \# (a^{-1} + b^{-1})^{-1} = a \# b$ if $a, b$ are strictly accretive.
  \end{proposition} \begin{proof}  First assume that $a, b$ are strictly accretive.  Then $a\# b$ is strictly accretive by Theorem \ref{straccr}. Also, if $c$ is invertible, then $c^{\ast}ac$ and $c^{\ast}bc$ are strictly accretive
(for example, if $a \geq \epsilon 1$ then $c^{\ast}ac \geq \epsilon c^* c$, and the 
latter is strictly positive. Hence $(c^{\ast}ac) \#(c^{\ast}bc)$  is strictly accretive, and 
its inverse, by the formula in the proof of Theorem \ref{straccr} is 
$$\frac{2}{\pi} \, \int_0^\infty \, (c^*(t a + \frac{1}{t} b)c)^{-1} \, \frac{dt}{t}
=  c^{-1} G^{-1} (c^{-1})^* ,$$  where $G = a \#b$.
so that $c^* G c = (c^{\ast}ac)\#(c^{\ast}bc)$.   

If $b$ is strictly accretive and $a$ is
invertible, then $c^{\ast}((a+ \epsilon 1)\#b) c=
(c^*(a+ \epsilon 1)c)\#(c^* b c)$.    Clearly the left side converges to $c^{\ast}(a\#b) c$ as $\epsilon \to 0^+$.
The  right side converges to $(c^*ac) \#(c^*bc)$  again by a slight variant of the proof of Corollaries \ref{accr},
\ref{accr2}, 
or Proposition   \ref{invsg}.

If  $a, b$ are merely accretive, and $a$ is invertible, then we have from the above that 
$$c^* (a \# (b + \epsilon 1)) c = (c^* a c)\#(c^* (b + \epsilon 1) c).$$
The  left side converges to $c^{\ast}(a\#b) c$ again.   If $w = c^* a c$ then 
$w^{-\frac{1}{2}} (c^* (b + \epsilon 1) c) w^{-\frac{1}{2}}$ and 
$w^{-\frac{1}{2}} (c^* b c + \epsilon 1) w^{-\frac{1}{2}}$ are both type $M$.
So by a variant of the last part of the proof of Corollary \ref{accr2} we have 
$$c^{\ast}(a\#b) c = \lim_{\epsilon \to 0^+} \, w^{\frac{1}{2}}
(w^{-\frac{1}{2}} (c^* b c + \epsilon 1) w^{-\frac{1}{2}})^{\frac{1}{2}} w^{-\frac{1}{2}}
= (c^*ac)\#(c^*bc).$$

If $a, b$ are strictly accretive then
 $a^{-1}$, $b^{-1}$, $(a+b)$ and $a^{-1}+b^{-1}$ are also strictly real positive by Lemma \ref{inv}.
Then the result follows as in the literature from the unicity of the solution to 
$x a^{-1} x = b$, and the relations $(a \# b) a^{-1} (a \# b)  = b$ and 
 $(a \# b) b^{-1} (a \# b)  = a$.
\end{proof}

Probably with slightly more work one can generalize
 the second assertion of the last result to the case that $a, b, a + b$ and 
$a^{-1} + b^{-1}$ are accretive 
and invertible.    Also, it seems possible that many of our results
on the geometric mean can be extended to the case when both operators are accretive
with no further conditions.  
We remark also that the well known algebraic-geometric-harmonic mean inequality
for the geometric mean of positive matrices fails badly for accretive matrices, as may be seen 
by considering $a = {\rm diag}(1+i,1-i), b = {\rm diag}(1-i,1+i)$ (here all three means are positive matrices, 
but violate--actually reverse--the usual inequalities).

  \section{The binomial  and Visser methods for the square root}  

We expect that the `binomial method' \begin{equation} \label{bin} 
X_{n+1} = \frac{1}{2}  (b + X_n^2) \; , \qquad X_0 = 0,\end{equation}
 and its variant the `Visser method' 
$$X_{n+1} = X_n + \alpha (a -X_n^2) \; , \qquad  X_0 = \frac{1}{2 \alpha} I,$$ work
in Banach algebras, under reasonable hypotheses.   Here $b = 1-a$, and it is expected that 
these schemes converge to $1 - a^{\frac{1}{2}}$ and $a^{\frac{1}{2}}$ respectively.   
Of course  it is well known that if $\Vert 1 - a \Vert \leq 1$
then the binomial {\em series} for $(1-(1-a))^t$ converges to $a^t$
(see e.g.\ \cite[Proposition 3.3]{BOZ}; this may have been first done in \cite{KV}).   However the binomial series is a little different from 
the binomial method above.   For operators $a$ on a Hilbert space
one can (more or less easily,
depending on the numerical range concerned)
 prove convergence results for the binomial method using the disk 
algebra functional calculus (coming from
von Neumann's inequality) or more generally
Crouzeix's functional calculus mentioned at the end of the introduction, which essentially reduces the computation
to one about scalars.  Then the matching Visser method result follows
by the usual substitution turning the binomial method into the 
Visser method (see the proof on \cite[p.\ 159]{High}, or Corollary \ref{visfa} below). 
This provides an effective iterative 
`polynomial approximation' for the square root of any operator in ${\mathfrak F}_A$ for 
an operator algebra $A$.   The following is intimately connected with the 
complex dynamics of the Mandelbrot set.   Indeed the scalar case of the `binomial method' (\ref{bin}) 
if we change variables $w = 2 x$, and let $c = b/2$, becomes the 
usual quadratic iteration $w_{n+1} = w_n^2 + c$ used to define
the Mandelbrot set.      The `main cardioid' for the binomial method is the set 
of attracting fixed points of $z \mapsto \frac{1}{2}(z^2 + b)$; and this may   be 
obtained almost identically to the Mandelbrot set case \cite[p.\ 15]{Beardon} 
from the open unit disk $D(0,1)$ by subtracting the latter from 1, then squaring all
elements, and then subtracting the resulting set from $1$.   

\begin{theorem} \label{binfa}   Let $b$ be a Hilbert space operator  with  numerical range  contained in 
a compact subset $E$ of  the
cardioid $2 z - z^2$ for $|z| < 1$, or more generally contained in 
the union of $E$ and the closed unit disk $\bar{D}(0,1)$. 
The binomial method {\rm  (\ref{bin})} applied to $b$ converges 
to $1 - a^{\frac{1}{2}}$, where $a = 1-b$.     As a special case, for any 
contraction $b$  on a Hilbert space (that is, $\Vert b \Vert \leq 1$), 
 the binomial method converges
to $1 - a^{\frac{1}{2}}$.
\end{theorem} \begin{proof}  Let ${\mathcal D}$ be the union of  the indicated disk and 
cardioid.   Define 
polynomials $q_n(z)$ on ${\mathcal D}$  by $q_0 = 0$ and $q_{n+1}(z)  = \frac{1}{2} (z + 
q_{n}(z)^2)$.  
Then $X_n = q_n(b)$, and we need to show that
$\Vert X_n - 1 + c \Vert \to 0$, where $c = (1-b)^{\frac{1}{2}}$.
 By the scalar case (see e.g.\ \cite[Theorem 6.14]{High}), $(q_n(z))$  converges pointwise on  
the interior of the cardioid to
$1 - (1-z)^{\frac{1}{2}}$, and the latter function is certainly analytic on some open subset of ${\mathcal D}$.  
Moreover $(q_n)$ is well known (and easily seen) to be uniformly bounded on the `main cardioid'.
(For example, this cardioid is bounded by $4$, and so if $z_n = q_n(z)$ for some $z$ in the cardioid, 
and if $|z_n|=4+a$, where $a>0$ then
$$|z_{n+1}| \geq \frac{1}{2} |z_n|^2-2=\frac{1}{2}  (4+a)^2-2=6+4a+\frac{1}{2} a^2>4 +4a.$$
By induction $|z_{n+k}|> 4 +4^k a \to \infty$ as $k \to \infty$, so $z_n  \to \infty$, which contradicts
one of the definitions of the Mandelbrot set--as the points whose iterations are bounded.)  
Thus by Vitali's theorem combined with Montel's theorem (see \cite[Section 3.3]{Beardon}), 
$(q_n)$ converges uniformly on any compact subset of the interior of the cardioid.   
We next show that $r_n(z) = q_n(z) - 1 + (1-z)^{\frac{1}{2}} \to 0$
uniformly on the disk.   We use an idea in the argument
for the scalar case from \cite{High}.    Let $w = (1-z)^{\frac{1}{2}}$.  We have 
$$r_{n+1}(z) = \frac{1}{2} ((z + 
q_{n}(z)^2) - 1 + w = \frac{1}{2}(q_n(z) + 1 - w) \, r_n(z) .$$
It is clear by induction that if $|z| \leq 1$ then $|q_n(z)| \leq 1$ for all $n$ (indeed if this is true for $n$ then by the binomial theorem we have 
$|q_{n+1}(z)| \leq \frac{1}{2} (1 + |q_n(z)|^2) \leq 1$).
For $|z| \leq 1$ we have 
$$|1-w| = |1 - (1-z)^{\frac{1}{2}}| 
\leq - \sum_{k=1}^\infty \, (-1)^k  \, { 1/2 \choose k} = 1 . $$
Hence $$|r_{n+1}(z) | \leq \frac{1}{2}(|q_n(z)| + |1-w|) |r_n(z) |
\leq |r_n(z)|.$$
Thus $(|r_n(z)|)$ is decreasing with pointwise limit $0$, so by Dini's theorem $(r_n)$ converges uniformly
on the unit disk.

Let   $E$ be the (closed) numerical range of $b$.    By the hypotheses
on $E$, and the facts just established, $(q_n)$ converges uniformly on $E$.
By Crouzeix's functional calculus (or we could use the disk algebra functional calculus coming from
von Neumann's inequality if $b$ is a contraction), for some constant $K$ we have 
$$\Vert q_m(a) - q_n(a) \Vert \leq K \Vert q_m - q_n \Vert_{E}, \quad m, n \in \Ndb.$$
Thus $(X_n)$ is Cauchy, and hence convergent to $w$ say.  We have $w = \frac{1}{2} (w^2 + b)$, 
so $a = (1-w)^2$.   We also note that any point in 
the spectrum of $w$  is a limit of $(q_k(z))$ for some $z \in E$, and hence equals
 $1 - (1-z)^{\frac{1}{2}}$.   Thus the spectrum of $1-w$ is the right half plane, and hence  
$1-w$ is the principal square root of $a$.
  \end{proof}

\begin{remark}  1) \ If $b$ in the last proof is a contraction then the part of the proof using Dini's theorem gives a
seemingly more controlled convergence, with the `error term' dominated by a decreasing null sequence.

\medskip

2) \ One may rephrase the last result in terms of subsets of the unit disk, instead of subsets of the cardioid.  Indeed
the homeomorphism between that disk and the cardioid mentioned before the theorem statement
gives a kind of passage between statements about 
$b$ and statements about  $1- a^{\frac{1}{2}} = 1 - (1-b)^{\frac{1}{2}}$.  
\end{remark}

\begin{corollary} \label{visfa}   Let $a$ be  
an operator on a Hilbert space with the numerical range of $1 - t^2 a$ (that is, 
$1-t^2 W(a)$) contained in $E \cup \bar{D}(0,1)$, where $E$ is as in the last theorem. 
Then the Visser method 
$X_{k+1} = X_k + \frac{t}{2} \, (a - X_k)^2$ with initial guess $X_0 = \frac{1}{t} I$, converges
to $a^{\frac{1}{2}}$.     In particular this holds if  $a \in {\mathfrak c}_{B(H)}$.  
\end{corollary} \begin{proof}   
By Theorem \ref{binfa} the binomial method  
applied to   $b = 1 - t^2 a$  gives a sequence $(B_n)$ converging  
to $1 - t \, a^{\frac{1}{2}}$.    So $\frac{1}{t} (1 - B_n) \to a^{\frac{1}{2}}$.   
However one can check that
$\frac{1}{t} (1 - B_n)$ coincides with  the $n$th step in the Visser method in the statement.

 If  $a \in {\mathfrak c}_{B(H)}$ then $\Vert 1 - t^2 a \Vert \leq 1$ for some $t > 0$, 
so we are in the special case that $b$ is a contraction in the last theorem.
 \end{proof}

\begin{remark}   There is probably a similar method for the $p$th root, and results similar to the two 
above in that case. \end{remark} 

\section{Newton's method for the pth root}

Newton's method for the $p$th root of $a$, for $p > 1$, is
$$X_{k+1} = \frac{1}{p} \, X_k \, ((p-1) I + X_k^{-p} \; a).$$   
With  $X_0 = I$ or $X_0 = \frac{1}{2}(a +I)$  this method
need not work for accretive matrices.  Indeed it fails even for some scalars in the right half
plane (see the discussion on page 178--179 of \cite{High}).    In the light of the scalar case, 
one would expect  that Newton's method for the $p$th root of $a$ with starting guess  $X_0 = I$
works with some restriction on $a$, such as that  the numerical range 
of $a$ should be in the region of convergence for the scalar case.
Let $${\mathcal D} = \{ z  \in \Cdb 
 : \, {\rm Re}(z) > 0 \; {\rm and} \; |z| \leq 1 + \epsilon \}
\cup \{ z  \in \Cdb 
 : \, {\rm Re}(z) > 0 \; {\rm and} \; |z-1| \leq 1 \}.$$

\begin{proposition} \label{newp}   Let $p > 1$ be an integer.   There exists $\epsilon > 0$ such that
for any Hilbert space operator $a$ with numerical range contained in the set
${\mathcal D}$ above, Newton's method for the $p$th root above, with initial point $X_0 = I$, 
converges to $a^{\frac{1}{p}}$.   \end{proposition}  \begin{proof}    
Define a sequence of rational functions $$q_{k+1}(z) =  \frac{1}{p} \, q_k(z) \, ((p-1)  + \frac{z}{q_k(z)^{p}}) \, , \qquad q_0 = 1,$$ 
for all $z$ where this makes sense (that is, where $q_k(z)$ is defined for all $k \in \Ndb$).
By  \cite[Lemma 2.11]{Ianz} there exists 
$\epsilon > 0$ such the sequence above does make sense if 
${\rm Re}(z) > 0$ and $|z| \leq 1 + \epsilon$, and 
 the  $(q_k)$ converges  to $z^{\frac{1}{p}}$  uniformly
on any compact subset of   $\{ z  \in \Cdb
 : \, {\rm Re}(z) > 0 \; {\rm and} \; |z| \leq 1 + \epsilon \}$.  

We next consider the set 
$K_1 = \{ z \in \bar{D}(1,1) : \; {\rm Re} \, z > \frac{1}{4} \}$.   On this set, if $(c_k)_{k\geq 2}$ is the 
sequence of positive numbers with sum $1$ from Lemma 1 in \cite[Section 3]{Guo}, we have 
 $|c_2 + c_3 (1-z)| \leq d < c_2 + c_3$ for some constant $d$.  This is
because $$(c_2 + c_3)^2 - |c_2 + c_3 (1-z)|^2 = 2 c_2 c_3 {\rm Re}(z) + c_3^2 (1- |1-z|^2)
\geq  \frac{1}{2} c_2 c_3 .$$ 
By the argument in the just mentioned lemma from \cite{Guo} we have
 $| 1 - z /q_1(z)^p | \leq \alpha$ for all $z \in K_1$, where
$\alpha = d + \sum_{k=4}^\infty c_k < 1$, 
and $$| 1 - z /q_n(z)^p | \leq | 1 - z /q_1(z)^p |^{2^{n-1}} \leq \alpha^{2^{n-1}}.$$
 The  sequence $(q_n(z))$ is well defined on $K_1$ by the argument
in \cite{Guo}.  And   $(q_n(z)^p)$ and therefore also $(q_n(z))$  is
uniformly bounded  on $K_1$, by a constant $M$ say, since
$q_n(z)^p = \frac{z}{1 - (1 - z /q_n(z)^p)}$.   Moreover, since $$|q_{n+1}(z) - q_n(z)|
= \frac{1}{p} |q_n(z)| |1 - z /q_n(z)^p| \leq  \frac{M}{p}  \alpha^{2^{n-1}} ,  \qquad z \in K_1, $$
it is easy to see that $(q_n)$ is uniformly Cauchy on $K_1$, so  uniformly  convergent.   
Thus the rational functions $(q_k)$ converge  to $z^{\frac{1}{p}}$  uniformly
on  any compact subset of the set ${\mathcal D}$ above.

Suppose that $W(a) \subset 
{\mathcal D}.$   By Crouzeix's 
functional calculus, for some constant $K$ we have 
$$\Vert q_m(a) - q_n(a) \Vert \leq K \Vert q_m - q_n \Vert_{W(a)}, \quad m, n \in \Ndb.$$
Thus $(X_n)$ is Cauchy, and hence convergent to $w$ say.   Since 
$$p X_k^{p-1} (X_{k+1} - X_k  +\frac{1}{p} \, X_k) = a,$$ in the limit we have 
$w^p =   a$.     We also note that by spectral theory any point in 
the spectrum of $w$  is a limit of $(q_k(z))$ for some $z \in E$ (namely $z = \chi(a)$ where $\chi$ is
a character of the closed algebra generated by $1$ and $a$), and hence equals
 $z^{\frac{1}{p}} \in S_{\frac{\pi}{2p}}$.   Thus $w$ is   the principal $p$th root of $a$.
\end{proof}  

\begin{remark}  (1) \ It is easy to see, by the same argument as the scalar case
(without needing Crouziex's functional calculus),
that the convergence in the last theorem is 
`quadratic' in the sense of the usual Newton $p$th root convergence for scalars.  
 
\medskip

(2) \   Experimentation shows that the polynomials $(q_n)$ in the last proof 
seem to converge uniformly on  the set  $\bar{{\mathcal D}}$.
If this is indeed the case then the last proof shows that Newton's method for the $p$th root
above converges to $a^{\frac{1}{p}}$ for any Hilbert space operator $a$ with numerical range contained in the set
$\bar{{\mathcal D}}$.

\medskip

(3) \  A similar idea of course shows that Newton's method for the $p$th root converges for
Hilbert space  operators with $T \geq 0$, no doubt a well known fact.   
Indeed the Newton iterates take place
in the unital $C^*$-algebra generated by $T$, which by Gelfand theory
may be taken to be $C(E)$ for a compact set $E \subset [0,\infty)$.
The functions $(q_n)$ in the last proof are easily seen to be decreasing (certainly for
$n \geq 2$) and hence
converge uniformly 
on $E$ by Dini's theorem.   Hence the Newton iterates are $\Vert q_n(T) - T^{\frac{1}{p}} 
\Vert = \Vert q_n - q \Vert_E \to 0$, where $q(t) = t^{\frac{1}{p}}$.   
\end{remark} 

\begin{proposition} \label{newpba}   Let $a$ be an element in a unital Banach algebra $A$ with
$\Vert 1- a \Vert < 1$.      
Let $p > 1$ be an integer.   Then  Newton's method for the $p$th root above, with initial point $X_0 = I$, 
converges to $a^{\frac{1}{p}}$ (indeed there exist  constants $\delta < 1$ and 
$C$ with $\Vert X_k  - a^{\frac{1}{p}} \Vert \leq C \delta^{2^{k}}$ for each $k$, and the latter
has quadratic convergence).    In particular, this is the case by Proposition {\rm \ref{Rsa}}  if $a$ is strictly accretive and
 $\Vert 1 - 2 a \Vert \leq 1$.   \end{proposition}  \begin{proof}   We follow the argument in \cite{Guo}, noting that 
Lemma 1 there holds with the same proof to show that the Newton sequence is well defined, and if
$\delta = \Vert 1 - a X_1^{-p} \Vert$ then 
$$\Vert 1 - a X_k^{-p} \Vert \leq \Vert (1 - a X_1^{-p})^{2^{k-1}} \Vert \leq \Vert 1 - a X_1^{-p} \Vert^{2^{k-1}}
= \delta^{2^{k-1}},$$
and $$\delta = 
\Vert 1 - a X_1^{-p} \Vert = \Vert \sum_{i=2}^\infty \, c_i (1-a)^i \Vert < \sum_{i=2}^\infty \, c_i = 1.$$
So $a X_k^{-p} \to 1$ rapidly with $k$
(the `error' is  dominated by $\delta^{2^{k-1}}$ where $\delta < 1$, which 
has quadratic convergence).     Hence  $X_k^p a^{-1} \to 1$ and $X_k^p \to a$, which means that
$\Vert X_k^p - 1 \Vert = \Vert X_k^p - a + a - 1 \Vert< 1$ for $k$ large.
Hence  $\Vert X_k - 1  \Vert < 1$ by e.g.\ \cite[Proposition 3.3]{BOZ} and its proof, so that $(X_k)$ is bounded.
It follows as in the proof of \cite[Theorem 5]{Guo}   (which is a result about the  scalar case, not operators)
$$\Vert X_{k+1} - X_k \Vert = \frac{1}{p} \Vert X_k (1 - a X_k^{-p}) \Vert \leq 
\frac{K}{p} \Vert 1 - a X_1^{-p} \Vert^{2^{k-1}} \leq \frac{K}{p} \delta^{2^{k-1}} $$
for a constant $K$.    
Hence $(X_n)$ is Cauchy, and we can finish the proof as in Proposition \ref{newp}.     
By the usual   triangle inequality argument we get  $\Vert X_{k+m} - X_k \Vert$ and 
hence $\Vert X_k  - a^{\frac{1}{p}} \Vert$, dominated  by $C \delta^{2^{k}}$, 
for a constant $C$.
\end{proof}

An important remark about other `starting values' in Newton's method: If $a$ is an element of a unital Banach algebra, and if Newtons method for the $p$th root of $a$ starting at  $X_0 = 1$ converges to $a^{1/p}$, then for any $\theta$ with 
$|\theta| \leq \pi$ say, Newtons method for  for the $p$th root of
$e^{i \theta} a$  starting at  $X_0 = e^{i \theta/p}$, converges.   Its limit is $e^{i \theta/p} \,  a^{1/p}$,  which is a $p$th
root of $e^{i \theta} a$.   Indeed the latter is the principal $p$th root of $e^{i \theta} 
a$ by Theorem
\ref{lrs}, if $W(e^{i \rho} a)$ contains 
no strictly negative numbers for all $\rho$ between $0$ and $\theta$ (including 
$0$ and $\theta$).    To see this define $Y_k = e^{i \frac{\theta}{p}} X_k$.  It is easy to see that  $(Y_k)$ coincides with the iterates in Newtons method 
for $e^{i \theta} a$  starting at  $X_0 = e^{i\theta/p}$.

\medskip

As in Iannazzo's paper \cite{Ianz} note that for any 
strictly accretive  Hilbert space  operator  $a$,
$b = a^{\frac{1}{2}}/\Vert a^{\frac{1}{2}} \Vert$ is also 
strictly accretive by e.g.\ a result on p.\ 181 of \cite{Haase},
 and has norm $\leq 1$.   So $W(b)$ lies in the set ${\mathcal D}$ considered in  Proposition \ref{newp} above.
Thus Proposition \ref{newp} applies to $b$, and so we can use Newton's method to find
$b^{\frac{1}{p}}$, from which   $a^{\frac{1}{p}}$ is easily recovered.

Another method to find the $p$th root of $a$ is to use the sign function studied 
in Section 3, in the way indicated in \cite[Section 3]{DHM}  in the matrix case.
In fact the beautiful arguments of \cite[Section 3]{DHM} 
go through with `eigenvalues' replaced
by `spectrum'.   As in that reference, if $p$ is odd we replace it by $2p$
and replace $a$ by $a^2$.   If $p$ is an integer multiple of $4$
we keep dividing it by $2$ and replacing $a$ by the square root of $a$,
until $p/2$ is odd.   We then set 
$$C = \left[ \begin{array}{ccccl} 0 & 1 & 0 & \cdots  & 0 \\
0 & 0 &  1 & \cdots  & 0 \\
\vdots & \vdots &\vdots &\vdots & \vdots \\
0 & 0 &  0 &   \cdots  & 1 \\
a & 0 &  0 &   \cdots  & 0
\end{array} \right],$$   (with the new $a$, if we had to change $a$ as above).  
Then $a^{\frac{1}{p}}$ can be read
off from the $1$-$2$ entry of sign$(c)$.   And in Section 3
we discussed the Newton method for sign$(c)$ and its convergence. 
We obtain, as in  \cite[Section 3]{DHM}:
 
\begin{theorem} \label{Newtpsign}  Suppose that
$a$ is an invertible element of a unital operator algebra with no negative numbers
in its spectrum, and let $v$ be the $1$-$2$ entry of ${\rm sign}(C)$ where $C$ is as above.
Then $a^{\frac{1}{p}} = \frac{p}{2 \sigma} v$ where $\sigma = 1 + 2
\sum_{k=1}^r \, \cos (\frac{2 \pi k}{p})$, and $r$ is the greatest  integer less than or equal to $p/4$. 
\end{theorem}

Halley's method for the $p$th root of $a$, for $p > 1$, is
$$X_{k+1} =X_k \; \frac{(p-1)X_k^p+(p+1)a}{(p+1)X_k^p+(p-1)a}, $$
with  $X_0 = I$.

In the light of the scalar case \cite{Ianz1},
one would expect  that this method should converge if  the numerical range
of $a$ is inside the region of convergence for the scalar case, which is a much
nicer region than that of the Newton method.   This is the case:

\begin{proposition} \label{halp}   Let $p > 1$ be an integer.   For any strictly 
accretive element $a$, Halley's method for the $p$th root above with initial point $X_0 = I$,
converges to $a^{\frac{1}{p}}$.
\end{proposition}

\begin{proof}
Define a sequence of rational functions $$q_{k+1}(z)  =q_k(z) \, 
\frac{(p-1)q_k(z)^p+(p+1)z}{(p+1)q_k(z)^p+(p-1)z}\, , \qquad q_0 = 1,$$
for all $z$ where this makes sense (that is, where $q_k(z)$ is well defined for all $k \in 
\Ndb$). By  \cite[Corollary 5.3]{Ianz1},  $q_k(z)$ is well defined for all $k \in 
\Ndb$ if $z\in 
\Hdb$, and  $q_k(z) \to z^{\frac{1}{p}}$ on $\Hdb$.
If $w_0(z)=z^{-\frac{1}{p}}$ and $w_k(z)=q_k(z)\cdot z^{-\frac{1}{p}}$, then
$$w_{k+1}(z)  =w_k(z) \frac{(p-1)w_k(z)^p+(p+1)}{(p+1)w_k(z)^p+(p-1)}\, ,$$
for any $z\in \Hdb$.
By  \cite[Lemma 5.2]{Ianz1} we have $\vert  \arg{(w_{k+1}(z))}\vert\leq \vert\arg{(
w_{k}(z))}\vert$, thus $\vert\arg{(w_{k+1}(z))}\vert< \frac{\pi}{2p}$ for any $z\in \Hdb.$
Therefore, $\vert\arg{(q_{k+1}(z))}\vert< \frac{\pi}{p},$ which means $q_k(z):\Hdb 
\rightarrow \Hdb$ for any $k\in\Ndb.$  By Montel's  normality criterion (see \cite[Section 3.3]{Beardon}), 
$(q_k(z) )$ is normal.  By Vitali's Theorem, $( q_k(z) )_{k=1}^{\infty}$  converges
locally 
uniformly  to $z^{\frac{1}{p}}$ on $\Hdb.$

Suppose that $W(a) \subset\Hdb.$   By Crouzeix's
functional calculus, for some constant $K$ we have
$$\Vert q_m(a) - q_n(a) \Vert \leq K \Vert q_m - q_n \Vert_{W(a)}, \quad m, n \in \Ndb.$$
Thus $(X_n)$ is Cauchy, and hence convergent to $w$ say.  If $\chi$ is
a character of the closed algebra generated by $1$ and $a$ then
$$\chi(w) = \lim_k \, \chi(X_k) = \lim_k \, q_k(\chi(a)) = \chi(a)^{1/p} \neq 0.$$
Thus $w$ is invertible. Since
$$(p+1)X_{k+1}X_k^{p}-(p-1) X_{k}^{p+1} = ((p+1)X_k-(p-1)X_{k+1}) \,  a,$$ in the 
limit we have
$w^p=  a$.       Now we may finish
as in Proposition \ref{newp}.
\end{proof}

\subsection*{Acknowledgment}   We thank Ilya Spitkovsky for assistance with understanding the result in \cite{MP}, and for other discussions.  We also thank the referee for his/her useful comments.
 This work had origin in the supervision of undergraduate research 
with Ms. Hasti Sajedi, a sophomore student without any linear algebra background beyond basic matrix computations.
She is in the process of 
checking numerically some issues and conjectures related to this topic.  The present paper represents the 
theoretical side of this investigation.   We thank Ms. Sajedi for many discussions, and for running MATLAB codes.
We hope to have a sequel paper more based on numerical examples.   
Our original motivation for this project
 was in hopes that the iterative methods considered here may be applied to deduce results about roots.

\end{document}